\documentclass[amscd]{amsart}
\usepackage[all]{xy}
\usepackage{hyperref,xcolor,geometry,amsmath,amsthm,mymacro}

\makeatletter
\let\@wraptoccontribs\wraptoccontribs
\makeatother

\theoremstyle{theorem}
\newtheorem{Thm}{Theorem}
\newtheorem{Cor}[Thm]{Corollary}
\newtheorem{Lem}[Thm]{Lemma}
\newtheorem{Prop}[Thm]{Proposition}

\theoremstyle{Definition}

\theoremstyle{remark}
\newtheorem{Rem}{Remark}

\def\square{\hbox{\vrule\vbox{\hrule\phantom{o}\hrule}\vrule}}

\newcommand{\cal}{\mathcal}
\newcommand{\imbed}{\hookrightarrow}
\renewcommand{\L}{{\cal L}}
\newcommand{\F}{{\cal F}}
\newcommand{\G}{{\cal G}}
\newcommand{\B}{{\cal B}}
\newcommand{\E}{{\cal E}}

\newcommand{\bbP}{{\mathbb P}}
\newcommand{\Aone}{{\mathbb A}^1}
\newcommand{\Zet}{{\mathbb Z}}
\providecommand\pcmd\providecommand
\pcmd{\gr}{\operatorname{gr}}
\pcmd{\Bun}{{\mathrm{Bun}}}
\pcmd{\cD}{{\mathcal D}}

\renewcommand{\MM}{{\mathcal M}}

\newcommand{\Hi}{{\mathrm{Hitch}}}
\newcommand{\DBu}{{{\mathcal D}_{\Bun}}}
\newcommand{\DBud}[1][d]{{{\mathcal D}_{\Bun^{#1}}}}
\newcommand{\DBudC}[1][d]{{{\mathcal D}_{\Bun^{#1}\x C}}}

\newcommand{\DBus}{(\DBu)_s}

\newcommand{\DBuu}{{{\mathcal D}_{\uBun}}}
\newcommand{\DBuud}{{{\mathcal D}_{\uBun^d}}}
\newcommand{\DBuus}{(\DBuu)_s}
\newcommand{\DBuuds}{(\DBuud)_s}

\newcommand{\DC}{{{\mathcal D}_{C}}}

\newcommand{\DBuC}{{{\mathcal D}_{\Bun\x C}}}

\newcommand{\Op}{{\rm Op}}

\pcmd{\Loc}{{\rm Loc}}

\newcommand{\Dmod}{\mathcal D\hmod}
\newcommand{\DBumod}{\Dmod_{\Bun}}
\newcommand{\DBuCmod}{\Dmod_{\Bun\x C}}
\newcommand{\DBudCmod}{\Dmod_{\Bun^d\x C}}
\pcmd{\hmod}{\textrm{-mod}}

\newcommand{\uBun}{\underline{\Bun}}

\newcommand{\LG}{^LG}

\renewcommand{\H}{{\mathcal H}}

\renewcommand{\ot}{\otimes}

\newcommand{\chark}{{\operatorname{char} k}}

\pcmd\red[1]{{\color{red}#1}}

\ncmd\Funiv{\F_{\mathrm{univ}}}
\ncmd\dinZ{\ensuremath{d\in\Zet}}
\ncmd\Ups\Upsilon
\ncmd\dis\displaystyle
\DefRm\univ
\ncmd\Om\Omega
\ncmd\xxx{\todo[XXX]}
\ncmd\HomC{\cHom_{/C}}
\ncmd\shtsr{\mathbin{\mathop{\tsr}\limits^!}}
\ncmd\DDBuC{D_{\Bun\x C}}
\ncmd\TCn{\cT_C^{\tsr n/2}}
\ncmd\RHomC{R{\HomC}}
\providecommand{\quash}[1]{}  
\ncmd\fm{\mathfrak m}
\ncmd\bbF{{\mathbb F}}
\title[Quantization of Hitchin integrable system]
{Quantization of Hitchin integrable system via positive  characteristic}

\dedicatory{To David Kazhdan with admiration}

\author{Roman Bezrukavnikov}
\author{Roman Travkin}
\contrib[With an appendix by]{Roman Bezrukavnikov, Tsao-Hsien Chen and Xinwen Zhu}
\begin{document}

\begin{abstract}
The main result of the seminal (unpublished) work of Beilinson-Drinfeld is the construction of an automorphic
sheaf corresponding to a local system which carries the additional structure of an oper. This is achieved by quantizing the Hitchin intergrable system.
In this note we show (in the case of $G=GL(n)$) that this result admits a short proof based on
positive characteristic methods.

In the appendix we study the restriction of the $p$-curvature ($p$-Hitchin) map  to the space of opers, we show it is finite and flat by checking
it is asymptotic to Frobenius at infinity. We speculate on the relation of this result to a conjecture of Frenkel, Etingof and Kazhdan on
the spectrum of global critically twisted differential operators on $\Bun_G$ acting on the space of $L^2$ sections of the bundle of half forms. 
\end{abstract}

\maketitle



\section{Introduction}
Geometric Langlands duality predicts the existence of an automorphic $\D$-module
$\MM_L$ on $\Bun_G$ attached to a (de~Rham) $\LG$-local system $L$. Here $G$, $\LG$
are reductive groups dual in the sense of Langlands, $\Bun_G$ is the moduli stack of $G$-bundles
on a complete smooth
irreducible curve $C$ and the local system $L$ on $C$ with structure group $\LG$
is assumed to be irreducible (i.e.,
 it does not admit a reduction to
 a proper parabolic subgroup).

In their celebrated unpublished work \cite{BD} Beilinson and Drinfeld explain that geometric Langlands duality can be thought of as a quantization of a natural duality for the Hitchin integrable systems
associated to two Langlands dual groups $G$, $\LG$. Furthermore, they present a construction of $\MM_L$ for a local system $L$ which carries
an additional structure of an {\em oper}, see \cite{BDop} for an introduction to this notion.
Their construction uses local to global arguments, it relies heavily on representation theory of affine
Lie algebras at the critical level.

In this note we describe, for $G=GL_n$, a much shorter construction  bypassing affine Lie algebra representations, relying instead on reduction to positive characteristic.

More precisely, we use the (easy) construction of automorphic $\D$-modules $\MM_L$ for a generic local system $L$ on a curve over a field of positive characteristic \cite{BB}. 
We establish the following property of this construction playing a crucial role in the present work:
 the (critically twisted) $\D$-module $\MM_L$ is generated by a global section if and only if the local system $L$ carries the structure of an oper.
The proof of this property is closely related to  ideas of \cite{BD}. 

As a formal consequence of that property we deduce that 
the (partially defined) geometric Langlands equivalence in positive characteristic of \cite{BB} sends the free
critically twisted $\D$-module to the structure sheaf of space opers, see section \ref{secfil} 

This allows us to show that
 global sections of the sheaf of critically twisted differential operators (i.e., global sections of the corresponding free rank~$1$ twisted $\D$-module) on $\Bun_n$ is a flat deformation of the ring of functions on the Hitchin base,
 by first doing it in positive characteristic and then deducing the general case by a standard argument. The construction of automorphic $\D$-modules corresponding to opers follows from this in view of the observation from section \ref{secfil}.

Let us mention that the theme started in \cite{BB}, where geometric Langlands duality was established for $GL_n$ local systems with a smooth spectral curve has been developed
 in \cite{Gr} where the case of not necessarily smooth spectral curves has been treated, in
  \cite{CZ}, \cite{CZ1} dealing with $G$ local systems for $G\ne GL_n$ and in \cite{Tr} where study of quantum geometric Langlands duality in that setting is initiated. 
  However, the present note is the first work (to the authors' knowledge) where this type of result is connected
 to the original setting of a characteristic zero base field.

\bigskip

The argument is based on an interesting property of  the so called $p$-Hitchin map ($p$-curvature) map $h_p$ restricted to the space of opers $\Op$ established in Appendix A\@.
Roughly speaking, it says that the map $h_p|_{\Op}$ is asymptotic to Frobenius at infinity (see Lemma \ref{pi_p} for a precise formulation).
This implies that $h_p|_{\Op}$ is finite, flat of degree $p^d$ where $d=\dim(\Bun_G)$ (Theorem \ref{a main}). 
In contrast with some other steps in the argument, this phenomenon is special to positive characteristic, it has not, to our knowledge, appeared in previous
works in geometric Langlands duality (a somewhat related concept of a dormant oper defined by Mochizuki \cite{Mo} has been studied in \cite{JP}, \cite{W} etc).
For the group $G=GL(1)$ the map $h_p|_{\Op}$ is described the Hasse-Witt matrix (cf. Remark \ref{rem_HW}), so one can consider $h_p|_{\Op}$ a {\em noncommutative} counterpart 
of the {\em Hasse-Witt matrix}. 

We refer the reader to Remark \ref{rem_EFK}  for a discussion of the formal parallel between a resulting description of the spectrum of global twisted differential
operators on $\Bun_G$ acting on half-forms and a conjecture of \cite{EFK} describing the spectrum of such a ring of differential operators acting on 
the Hilbert space of $L^2$ sections of half forms on the moduli space of bundles on a complex curve. 

\bigskip

We finish the Introduction with a technical remark.
Below we use a general construction, the derived category of asymptotic $\D$-modules on stacks. Here by an asymptotic $\D$-module on a smooth
algebraic variety $X$ we mean a sheaf of modules over the sheaf of rings $D_\hbar(X)$, the
sheaf of Rees algebras corresponding to the sheaf of filtered algebras $D(X)$. Thus $D_\hbar(X)$
is a flat sheaf of rings over polynomials in $\hbar$, such that $D_\hbar(X)/(\hbar -1)=D(X)$,
$D_\hbar (X)/\hbar = \O(T^*X)$. We refer to \cite{L} for a discussion of standard functors on
the derived category of $D_\hbar$-modules. The proof of Proposition \ref{filtr} relies on an extension
of this theory to smooth algebraic stacks over a field of an arbitrary characteristic which does not seem to be documented in the literature.

The proof of Lemma \ref{OtC} uses rudimentary theory of derived stacks (the only derived stacks appearing
here are derived fiber products of ordinary stacks), we use \cite{T} as a general reference for their basic properties.

\medskip

{\bf Acknowledgements.} It is an honor for us to dedicate this work to David Kazhdan whose influence on the subject and on the authors can not be overestimated.

The first author would like to  thank G.~Laumon who several years ago has
asked him if the result of \cite{BB} can be applied to the characteristic zero setting.
R.B. was supported by  NSF grant  DMS-2101507. 
T.-H.C. is supported by NSF grants DMS-2001257 and DMS-2143722.
X.Z.  is supported by NSF grants DMS-1902239 and
DMS-2200940.

\section{Notations and statement of the main result}
We mostly work over a 
field $k$ of characteristic different from 2, we fix
a complete smooth geometrically connected curve $C$ over $k$ of genus at least 2. Let $G=GL_n$, $\Bun$ will denote the moduli stack
of rank $n$ vector bundles over $C$; let $\Bun^d$ denote the component of parametrizing bundles of degree $d$.

Recall the stack $\uBun$ with a map $\Bun\to \uBun$ which is a $\Gm$-gerbe,
\cite[\S 4.6]{BB}. The categories of coherent sheaves and $\D$-modules on $\uBun$ and $\Bun$ are closely related, while
$\uBun$ has the advantage of being {\em good} in the sense of \cite{BD}; we let $\uBun^d
\subset \uBun$ be the image of $\Bun^d$.

Let $\DBumod$ be the category of twisted $\D$-modules on $\Bun$, where the class of the twisting
equals  half of the class corresponding to the canonical line bundle on $\Bun$
and similarly for $\uBun$.
Notice that a square root of the canonical line bundle on $\Bun$ is known to exist, thus this category is equivalent to the category of $\D$-modules on $\Bun$, respectively
$\uBun$.
Let $\DBu\in \DBumod$ denote the sheaf of twisted differential operators with the same twisting.
(Notice that $\DBu$ is {\em not} a sheaf of rings on $\Bun$, see \cite[Sect.~1.1.3]{BD}.)
Let also $D_{\Bun}$ be the derived category
of $\D$-modules on $\Bun$. Thus $\DBumod$ is the heart of the natural $t$-structure on $D_{\Bun}$.
We also use similar notations with $\Bun$ replaced by $\uBun$.

 Let $\Op$ denote the space of  {\em marked} opers, see \cite{BDop} for a general introduction
 to this notion, see also the definition (in the version we use) below before Corollary \ref{filtr_cor}.

The main result of this note is the following

\begin{Thm}\label{main}
\begin{enumerate}
\rncmd\theenumi{\alph{enumi}}
\item For every $d\in {\mathbb{Z}}$ we have a canonical isomorphism $\Gamma(\uBun^d,\DBuu)\cong \Gamma(\O_\Op)$.

\item For a point $x\in \Op$ corresponding to a local system $\L_x$ the $\D$-module $\DBu\otimes_{\O_\Op} k_x$ is a Hecke eigenmodule with respect to the  local system $\L_x$.
\end{enumerate}
\end{Thm}

Our strategy is to first establish the Theorem when $k$ has prime characteristic using the result
of \cite{BB} and then formally deduce the characteristic zero case.

\section{Hecke functor and filtrations}\label{secfil}
In this section we introduce a filtration on the image of the free $\D$-module under the  Hecke functor.
This is done uniformly in all characteristics by a direct argument independent of \cite{BB}; the idea is close in spirit to
 \cite[\S 5.5]{BD}.

We now recall the definition of
 the Hecke functor corresponding to the tautological representation of $GL_n$.

Let  $\H$ be the stack  parametrizing inclusions of vector bundles of rank~$n$,
  $\E_1\imbed \E_2$, whose cokernel is of length one.
We define $q_1,q_2: \H\to \Bun$ and $q_C: \H\to C$ by $q_i:(\E_1\subset \E_2)\mapsto \E_i$
and $q_C: (\E_1\subset \E_2)\mapsto x$ where $x$ is determined by the short exact sequence
$0\to \E_1\to \E_2\to \O_x\to 0.$


We also consider the projections   $q_{i,C}=\angs{q_i,q_C}\colon \H\to\Bun\x C\ (i=1,2)$.
 Notice that both $q_{1,C}$ and $q_{2,C}$ are ${\mathbb P}^{n-1}$ bundles.

The following statement is standard.

\begin{Lem}\label{rel_tan}
The relative tangent bundles $\cT_1$, $\cT_2$ for the maps
  $q_{1,C}$, $q_{2,C}$ admit a nondegenerate pairing
$\cT_1\tsr \cT_2\to q_C^*\cT_C= q_C^*\Omega_C^*$. 

\end{Lem}

We let $\DC$ denote the sheaf of twisted differential operators on $C$ that act on sections of the line
bundle $\TCn$; here for odd $n$ we use the choice of a square root of
the canonical bundle $\Omega_C$.
We let $D_C$, $\DDBuC$ denote the corresponding derived categories of twisted
 $\D$-modules.

Using the Lemma it is easy to see that
 the sum of pull-backs under $q_1$, $q_2$ and $q_C$
of the above twisting classes equals the class of a line bundle (we will
use a more precise information about this class below); thus
we can define the {\em Hecke functor} between the categories of twisted $D$-modules: $H=(q_{2,C})_* q_1^![- n]: D_{\Bun}\to
\DDBuC$.
We will also need the dual Hecke functor, given by 
$H\chk=(q_{1,C})_* q_2^![-n]: D_{\Bun}\to
\DDBuC$.

Notice that we used smoothness of the Hecke stack to define the Hecke functor uniformly in all characteristics, a direct analogue of this definition for arbitrary reductive group $G$ works for Hecke functors corresponding to minuscule coweights only: doing it for more general weights requires intersection cohomology sheaves which do not admit a direct
generalization to positive characteristic

The functors $H$ and $H\chk$ are adjoint in the following sense: there is a functorial isomorphism 
\begin{equation}\label{adjH}
    \Hom_{\DDBuC}(M\bx(\TCn  \tsr N),H(M'))
    \isoto \Hom_{\DDBuC}(H\chk(M) \shtsr p_C^!N,M'\bx \TCn)
\end{equation} 
for all
 $M,M'\in D_{\Bun}$ and $N\in D_C\hmod$,  where $p_C$ is the projection $\Bun\x C\to C$ and $D_C\hmod$ is the category of \emph{untwisted} $\D$-modules  on~ $C$
  (here we use a choice of a square root of $\cT_C$).  
The above adjunction can also be characterized as follows.  Let us consider the bifunctor $\DDBuC^{\on{op}} \x \DDBuC \to D(C)$,  where $D(C)$ is the derived category of complexes of (untwisted) $D_C$-modules, given by  $(M,M') \mapsto \HomC(M,M') := p_{C,\blt} \cHom_{D_{(\Bun\x C)/C}}(M,M')$.

The functor $\HomC$ is characterized by the following adjunction:
   \[  \Hom_{D(C)}(N,\HomC(M_1,M_2)) \isoto \Hom_{\DDBuC} (M_1\shtsr p_C^! N,M_2)  \]
Using $\HomC$, the isomorphism~\eqref{adjH} can be rewritten as follows: 
\begin{equation}\label{adjH-Hom}
  \HomC(M\bx\TCn  ,H(M'))  \isoto  \HomC(H\chk(M) ,M'\bx \TCn)
\end{equation} 

Let $(\DC)_{\leq n}$ denote the term of the standard filtration by the order of a differential operator.

\begin{Prop}\label{filtr}
The object  $H(\DBu)$ lies in the abelian category $\DBuCmod$. Furthermore, the
$\DBuC$ module $H(\DBu)$  admits a canonical
map $c\colon \DBuC \to H(\DBu)$, 
such that the
restriction of $c$
to $\DBu \boxtimes (\DC)_{<n}$
is an isomorphism of quasicoherent sheaves.

\end{Prop}

\begin{proof}
The proof uses the category of ``asymptotic'' $\D$-modules $\cD_h$ (cf.~\cite{L}).
Recall that for a smooth variety $X$ the sheaf of rings $\cD_h(X)$ is obtained from the filtered
sheaf of rings $\cD(X)$ (differential operators on $X$) by the Rees construction. Thus $\cD_h$
is a sheaf of graded rings on $X$ with a central section $h$, such that $\cD/h$ is isomorphic
to the sheaf  $\O_{T^*X}$, while the localization $\cD_{(h)}$ is isomorphic to $\cD(X)[h,h^{-1}]$, one then considers
the (derived) category of sheaves of graded modules. A similar construction
applies  to twisted differential operators on stacks in the sense of \cite{BD}.
Notice that the subcategory of $h$-torsion free coherent asymptotic $\D$-modules is equivalent to the category of coherent $\D$-modules equipped with a good filtration.

The push-forward and pull-back functors are defined for ``asymptotic'' (twisted) $\D$-modules in a way compatible with the natural (derived) functor from the category of $D_h$ modules to that of $\D$-modules (quotient by $h-1$), see \cite{L}.
Moreover, as shown in {\em loc.\ cit.} the pull-back functor is compatible under the specialization at $h=0$ with the functor
between the derived categories of coherent sheaves on the cotangent bundles given by the natural correspondence; while the push-forward functor under a proper morphism $f:X\to Y$ is compatible with the
functor given by the natural correspondence up to twist by the line bundle $K_X\otimes K_Y^{-1}$.

This theory can be generalized to smooth algebraic stacks. 
 Notice that in the stack case the cotangent bundles and/or the relevant fiber products may have to be taken in the category of derived stacks.

We now proceed to spell this out in the present case.
Let $\Hi=T^*\Bun$. Recall that $\Hi$ is the stack parametrizing {\em Higgs fields},
 i.e., pairs $(\E,\phi)$ where
$\E$ is a rank $n$ bundle on $C$ and $\phi\in H^0(End(\E)\otimes \Omega_C)$.

Let $\H_{\Hi}$ denote the {\em Hitchin Hecke stack}  parametrizing triples $(\E_1\subset \E_2,\phi) $
where $(\E_1\subset \E_2)\in \H$ and the Higgs field
$\phi:\E_2\to \E_2\otimes \Omega_C$ satisfies $\phi(\E_1)\ctd\E_1\tsr\Omega_C$.

 We have $pr_1, pr_2:\H_\Hi\to \Hi$,
$pr_i:(\E_1\subset \E_2,\phi)\mapsto (\E_i,\phi|_{\E_i})$ and $pr_C: \H_\Hi\to T^*C$
sending $(\E_1\subset \E_2,\phi)$ to $(x,\xi)$ where $x$ and $\xi$ are determined by
$\O_x\cong \E_2/\E_1$, $(\phi -\xi \otimes \id_{\E_2})(\E_2)\ctd \E_1\otimes \Omega_C$.
Here $\H_{\Hi}$ is considered as a derived stack.

The free rank one $\D$-module $\DBu\in\DBumod$ equipped with the standard filtration determines an object
in the category of asymptotic $\D$-modules on $\Bun$ which we denote by
$\wtld{\DBu}$. Applying the Hecke functor $H_{asymp}=(q_{2,C})_* q_1^*$ to this object
(where $H_{asymp}$ denotes the Hecke functor on the category of asymptotic $\D$-modules)
we get an object $\wtld{H(\DBu)}$.

Using the above compatibility of pull-back and push-forward functors with the specialization
at $h=0$ and Lemma \ref{rel_tan}, one checks
 that the corresponding coherent sheaf $\wtld{H(\DBu)}\Ltsr _{k[h]} k$ is given by:
\begin{equation}\label{wtH}
\wtld{H(\DBu)}\Ltsr _{k[h]} k\cong (pr_2\x pr_C)_* pr_1^*(\O_\Hi).
\end{equation}

Let $\tilde C_\univ\subset \Hi \x T^*C$ be the {\em universal spectral curve},
i.e., $\tilde C_\univ$ parametrizes the data of $(\E, \phi; x, \xi)$ where $x\in C$, $\xi\in T^*C|_x$
such that $\det(\phi|_x - \xi \otimes Id)=0$.

\begin{Lem}\label{OtC}
We have $(pr_2\x pr_C)_* pr_1^*(\O_\Hi)\cong \O_{\tilde C_\univ}$.
\end{Lem}

\begin{proof} Consider subvarieties $S\subset \gl_n\x \Aone$ and
$\tilde S\subset \gl_n \x \bbP^{n-1} \x \Aone$ given by:
$S=\{(x,t)\ |\ \det(x-t\cdot Id)=0\}$, $\tilde S= \{ (x,l, t)\ |\ x|_l=t\cdot Id\}$.
Let $\pi:\tilde S \to \gl_n \x \Aone$ be the natural projection.
It is a standard fact that
\begin{equation}\label{piO}
\pi_*(\O_{\tilde S})\cong \O_S.
\end{equation}

Consider the  natural ``evaluation'' map $\Hi \x T^*C \to (\gl_n \x \Aone)/(GL_n \x \Gm)$
(where $GL_n$ acts on $\gl_n$ by the adjoint action and $\Gm$ acts on $\Aone$ by dialtions).
Namely, the line bundle $\Om_C$ defines a $\Gm$-torsor on $C$, and also we have a tautological $GL_n$-torsor on $\Bun\x C$ corresponding to the universal rank~$n$ vector bundle $\cE_\univ$ on $\Bun\x C$.  We pull back both torsors to $\Hi\x T^*C$ and take their fiber product, thus getting a $(GL_n\x\Gm)$-torsor $\cP$ on $\Hi\x T^*C$.  Further, we have canonical sections of the pullbacks of $\Om_C$ to $T^*C$ and of $\cEnd\cE_\univ \tsr\pr_C^*\Om_C$ to $\Hi\x C$.  Using the tautological trivializations of the pullbacks of both bundles to $\cP$, pulling back the canonical sections defines a $(GL_n\x\Gm)$-equivariant map $\cP\to\gl_n\x\Aone$ which, passing to the quotient stacks, gives our map  $\Hi \x T^*C \to (\gl_n \x \Aone)/(GL_n \x \Gm)$.

We have natural isomorphisms:
$$\H_{\Hi}\cong (\tilde S/(GL_n\x \Gm)\x _{(\gl_n \x \Aone)/(GL_n \x \Gm)}
(\Hi\x T^*C);$$
$$\tilde C_\univ\cong (S/(GL_n\x \Gm)\x _{(\gl_n \x \Aone)/(GL_n \x \Gm)}
(\Hi\x T^*C).$$

Here both fiber products are understood to be derived, thus both formulas are isomorphisms of
derived stacks.

Thus base change isomorphism applies, see \cite[\S 3.1 ]{T} and \cite[Proposition 1.4]{T1}.\footnote{In fact,
the base change isomorphism is stated in {\em loc. cit.} for derived product of (derived) schemes rather than 
derived stacks. The fiber product we presently consider is locally a quotient of a fiber product of schemes
by an action of an affine algebraic group; moreover, it is a union of open (derived) substacks of this form.
Thus the base change isomorphism in this case follows from {\em loc. cit.}.}

Thus  the lemma follows from \eqref{wtH}~and~\eqref{piO}. \end{proof}

\begin{Rem} a) Using Koszul resolution one can write down an explicit sheaf of DG-algebras on
$\Bun\x C$, 
such that its derived
category of sheaves  of modules is identified with the derived coherent sheaves category
of the derived fiber product.
 Thus one can work with these categories without invoking the
general theory of derived stacks, see \cite{Riche} where this approach is spelled out in another context.

b) We do not know if the derived fiber product in the last displayed formulas is essentially derived, i.e., if some of higher $Tor$'s
between the structure sheaves of the two factors over the structure sheaf of the base are nonzero.
If this is not the case the isomorphisms can be understood as isomorphisms of ordinary
stacks.
\end{Rem}

We are now ready to finish the proof of the Proposition. Recall that an object $\cM$
in the derived category of asymptotic $\D$-modules on a 
stack $X$ such that the induced object
$M=\cM\Ltsr_{k[h]}k\in D^b(Coh^{\Gm}(T^*X)$
lies in homological degree zero amounts
to a $\D$-module with a good filtration whose associated graded is isomorphic to $M$.
Thus comparing \eqref{wtH} with Lemma \ref{OtC} we see that $H(\DBu)$ is a $\D$-module
with a good filtration whose associated graded is isomorphic to $\O_{\tilde C_\univ}$.
Since the latter coherent sheaf is cyclic, we see that $H(\D)$ is a cyclic $\D$-module with a canonical generator. Since the sheaf of regular functions on $\Hi \x T^*C$ which have degree
less than $n$ along the fibers of projection  $\Hi \x T^*C \to \Hi \x C$ maps isomorphically
to $\O_{\tilde C_\univ}$, the Proposition follows. \end{proof}

Recall the ring of twisted differential operators $\D_C$ introduced after Lemma \ref{rel_tan}.
By an {\em oper} we will understand an $\O$-coherent $\D_C$ module $\O$ of rank $n$
which has a good filtration whose associated graded is isomorphic to $\gr(\D_C)_{< n}$.
Choosing a theta-characteristic (i.e., a square root of the cotangent bundle)
we can identify the category of $\D_C$-modules with the category of $\D$-modules and connect this with the standard definition of a (\emph{marked}) oper. It is standard that opers in this sense are parametrized
by a variety which we will denote $\Op$.

\begin{Cor}\label{filtr_cor}
Assume that $M\in D^b(\DBumod)$ satisfies the Hecke eigenproperty with respect to a local
system $\L\in \D_C\hmod$. Assume\footnote{The first assumption holds
 automatically if $\chark = 0$. If $\chark=p>0$, then
 $\deg(\L) \equiv (1-g)n(n-1) \pmod p$, and $\deg(\L)$ is determined by the character  by which $\Gm$
 acts on the quasicoherent sheaf underlying $M$ (where we use that $\Bun$ is a $\Gm$ gerbe
 over $\uBun$).  It is not hard to see that
 $\RHom_{\DBumod}(\DBu,M)=0$ automatically unless $\deg (\L)=(1-g)n(n-1)$.}
that $\L$ has degree $(1-g)n(n-1)$
and  does not admit an oper structure. Then
$$\RHom_{\DBumod}(\DBu,M)=0.$$
\end{Cor}

\begin{proof} The proof will proceed by contradiction.
Let $\L$ be the corresponding local system and
set $V:=\RHom_{\DBumod}(\DBu,M) $, thus $V\ne 0$ by assumption.

Consider $\pr^\O_{C*}(H(M))$, 
the sheaf direct image of
 the $\D$-module $H(M)$ under the projection
$\pr_C:\Bun_n\x C\to C$.
Then the Hecke eigen-property of $M$ shows that
$$ \pr^\O_{C*}(H(M))\cong V\otimes \L.$$

On the other hand, 
Proposition \ref{filtr} and isomorphism~\eqref{adjH-Hom} imply that the object in the derived
category of quasicoherent sheaves $\pr^\O_{C*}(H(M))$
satisfies:
\[\begin{split}
\pr^\O_{C*}(H(M))
  & \cong   \RHomC(\DBu\bx\TCn,H(M)) \tsr \TCn  \\
  & \cong   \RHomC(H\chk(\DBu),M\bx\TCn) \tsr \TCn  \\
  & \cong   \pr_{C*}(R\cHom(\DBu\bx(\DC)_{<n},M\bx\TCn)) \tsr \TCn  \\
  & \cong  \RHom_{\DBumod}(\DBu,M) \tsr_\k \cHom_{\O_C}((\DC)_{<n},\TCn)  \tsr_{\O_C} \TCn \\
  & \cong V\tsr(\DC)_{<n}  
\end{split}\]
where $H\chk$ is the adjoint Hecke functor $H\chk\colon\Dmod_\Bun\to\Dmod_{\Bun\x C}$.

Comparing the two displayed isomorphisms we see that ${\rm oblv}^\D_\O(\L)$
admits an injective map into $(\DC)_{<n}$ as a coherent sheaf.
Since an injective map between coherent sheaves on a curve
having the same degree and
the same generic rank has to be an isomorphism, we see that $\L$ has an oper structure. \end{proof}

\section{Proof of the main theorem in the case $\chark=p>0$}
It is easy to deduce the assertion of the theorem for $k$ from the assertion for the algebraic closure of~$k$, so we assume for simplicity that $k$ is algebraically closed.

Recall that $\Hi=T^*\Bun$ and
$\tilde C_\univ$ is the universal spectral curve. Let  $h:\Hi\to B$ be the Hitchin map and
$\pi:\tilde C_\univ\to B$ be the projection.
Let $B_r\supset B_s$
be the open subsets in the Hitchin base $B$  parametrizing the points  $x\in B$ such that
the fiber $\pi^{-1}(x)$ is reduced, respectively, smooth.

In this section we assume that the base field $k$ has prime characteristic $p$.
Then $\DBu$ can be thought of as a sheaf over $\Hi^{(1)}$, where the superscript denotes
the Frobenius twist.

Let $\Loc$ denote the moduli stack of $\DC$-modules which are locally free of rank $n$
as an $\O$-module. 
Recall \cite{BB} that we have the {\em Frobenius-Hitchin} map $h_p\colon \Loc\to B^{(1)}$; for example,
for $x\in B_s$ the fiber of $h$ over $x$ is the abelian algebraic group $Pic(\tilde C_x)$, while
the fiber of $\pi_p$ over $x^{(1)}\in B^{(1)}$ is the torsor over the abelian algebraic group
$Pic(\tilde C_x)^{(1)}$ (here $x^{(1)}$ denotes the image of $x$ under Frobenius).

We will need the following result proven in the Appendix.

\begin{Prop}\label{finite_flat}
The composition $\Op\to \Loc \xra{\pi_p} B^{(1)}$ is a flat finite map of degree
$p^N$, where $N=\dim (\uBun)$.
\end{Prop}

Set $\Hi_s=h^{-1}(B_s)$, $\Loc_s=\pi_p^{-1}(B_s^{(1)})$, $\Op _s=\Op\x _{B^{(1)}} B_s^{(1)}$,
$\DBus= \DBu|_{\Hi_s^{(1)}}$,
where in the last expression we use the same notation $\DBu$ for the object in $D_{\Bun}$
and the corresponding sheaf on $\Hi^{(1)}$.

Recall that the main result of \cite{BB} is an equivalence\footnote{In fact, this is a version
of the equivalence constructed in {\em loc.\ cit.}: there the stack $\Bun$ and ordinary $\D$-modules
are considered instead
of $\uBun$ and twisted $\D$-modules. It is not hard to deduce that version from the result of
{\em loc.\ cit}.  \nopagebreak Note in particular that replacing  $\Bun$ by  $\uBun$ in the left-hand side corresponds to restricting to one of the connected components of $\Loc$ in the right-hand side.}
$$D^b(\DBuus\hmod_{coh}) \cong D^b(Coh(\Loc_s^{(1-g)n(n-1)})),$$
where $\hmod_{coh}$ stands for the category of coherent sheaves of modules, and $\Loc_s^{(1-g)n(n-1)})$ stands for the component of $\Loc_s$ classifying $\D_C$-modules (locally free of rank~$n$ over $\O_C$ with smooth $p$-spectral curve) whose underlying $\O_C$-module has degree $(1-g)n(n-1)$.
We let $\Phi$ denote that equivalence.

The first step in the proof of the Theorem is the following

\begin{Prop}\label{Ops}
We have $\Phi(\DBuus)\cong (\Op_s\to \Loc_s)_*(\O_{\Op_s})$.
\end{Prop}

\begin{proof}
It is easy to see that the tautological map from the space  $\Op$ of marked opers to
$\Loc$ is a composition of the map $\Op\to \Op/\Gm$ and a closed embedding
$\Op/\Gm\to \Loc$, where we use the trivial action of $\Gm$ on $\Op$.
Thus the direct image of $\O_{\Op}$ to $\Loc$ decomposes as a direct sum indexed by characters
of $\Gm$; we claim that the summand corresponding to the character $t\mapsto t^d$
is canonically isomorphic to $\Phi((\DBuud)_s))$.

 It follows from Corollary \ref{filtr_cor} that the complex $\Phi((\DBu)_s))$ is supported
on $\Op_s/\Gm$.

From Proposition \ref{finite_flat} we see that its support is finite over
the base $B^{(1)}$. Since Fourier-Mukai transform is exact on sheaves with finite support,
sending such a sheaf of length $r$ into a vector bundle of rank $r$,
we see that $\Phi((\DBuud)_s)$ is concentrated in homological degree zero; moreover, its pull-back to $\Op_s$
is flat of rank $p^N$ as a module over $\O_{B_s^{(1)}}$.

We claim that $\Phi((\DBuud)_s)$ is scheme theoretically supported on $\Op_s/\Gm$.

First of all, $\Phi((\DBuud)_s)$ is torsion free as an $\O(\B^{(1)})$ module.
Thus it suffices to check this claim over the generic point of $\B^{(1)}$.

To see this notice that a coherent sheaf on $\Loc$ which is generically set theoretically but not scheme theoretically supported on $\Op/\Gm$ would need to have length greater than one at the generic
point of $\Op$. Then its direct image to $B^{(1)}$ would have generic rank greater than
$p^N$, which contradicts the second paragraph of the proof.

Now, flatness of $\Phi((\DBuud)_s)$ over $\O_{B^{(1)}_s}$ implies it's a Cohen-Macaulay
module over $\O_{\Op_s}$. Since $\Op_s$ is smooth, it is actually a locally free module,
since its degree over $\O_{B_s^{(1)}}$ equals that of $\O_{\Op_s}$, we conclude that the pull-back of
 $\Phi((\DBuud)_s)$ to $\Op_s$ is a line bundle on  $\Op_s$. Since $\Op_s$ is an open subvariety in $\Op$ which
 is isomorphic to the affine space, every line bundle on $\Op_s$ is trivial.
 Since $\Gm$ acts by the character $t\mapsto t^d$ on the $\Phi(\F)$ for $\F$ supported
 on $\uBun^d$, we get the statement.
\end{proof}

The Proposition implies that for all $d\in \Zet$
\begin{equation}\label{Gamma}
A_d:=\Gamma(\DBud)\subset \Gamma((\DBud)_s)\cong End((\DBud)_s)\cong \Gamma(\O_{\Op_s})
\end{equation}
is a commutative algebra.


Fix some \dinZ.
Proposition \ref{filtr} allows one to construct a family of opers on $C$ parametrized
by  $Spec(A_d)$. The family can be described as a $\DC$-module $\Funiv^d$
with an $A_d$ action and a filtration which is flat over
$ A_d$ and such that $\Funiv^d \otimes _{A_d} k_x$ is an oper
for every $x\in Spec(A_d)$ and is constructed as follows.
The sheaf $\Funiv^d$ is the (sheaf theoretic) direct image to the second factor
of  the $\DBud$-module $H(\DBud[d-1])$, equipped with the natural $A_d$ action
and the filtration coming from Proposition~\ref{filtr}.
The $A_d$-action is defined by presenting $\Funiv^d$ as the pushforward to $C$ of the local Hom from $\D_{\Bun^d\x C}=\D_{\Bun^d}\boxtimes\D_C$ to $H(\DBud[d-1])$ and using $A_d$-action on the first argument.  

Actually the definition of this action can be generalized to a functor $\Upsilon\colon \Dmod_{\Bun^d\x C}\to (A_d\otimes \D_C)\hmod$ given by  pushforward to~$C$,  and $\Funiv^d=\Upsilon(H(\DBud[d-1]))$.  Applying the functor $\Upsilon$ to the map $c$ from Proposition~\ref{filtr}, we get a map $\Upsilon(c)\colon \Upsilon(\DBudC) = A_d\otimes \D_C \to \Funiv^d$ which restricts to an isomorphism of $A_d\otimes\O_C$-modules $A_d\otimes(\D_C)_{<n} \to \Funiv^d$.  From this it is straightforward to see that $\Funiv^d$ defines an $A_d$-family of opers.

 Thus we get a map $\Pi: Spec(A_d) \to \Op$.  We will show that it is an isomorphism.

 It is easy to deduce from the Hecke eigen-property for the equivalence $\Phi$ that  base-change of $\Pi$ from $B^{(1)}$ to $B_s^{(1)}$ coincides with the (dual of) isomorphism \eqref{Gamma}.  From this we see that the composition \[
    \O(\Op) \xrightarrow{\Pi^*} A_d\subset\Gamma(\DBuuds) \overset{\text{Prop.~\ref{Ops}}}\cong \O(\Op_s)
 \]
 is the natural inclusion.  Thus $A_d$ is isomorphic to subalgebra of $\O(\Op_s)$ containing $\O(\Op)$.  Since $\Op$ is normal (it is isomorphic to an affine space: indeed, by Lemma~\ref{gr}, part~(1), $\Op$ is affine and there is a filtration on $\O(\Op)$ with $\gr\O(\Op)\cong\O(B)$ which is a polynomial algebra; see also \cite[Proposition~3.2.3]{JP}), it would suffice to show that $A_d$ is finitely generated as a module over $\O(\Op)$.  We will in fact prove finite generation over a smaller algebra.




\begin{Lem}\label{finite}
$A_d$ is a finitely generated torsion free module over $\O(B^{(1)})$.
\end{Lem}

\begin{proof}
  Consider the filtration on $\DBu$ and the induced one on $A_d$ by degree of differential operator.   Then we have $\gr \DBud=(T^*\Bun^d \to \Bun^d)_*\O_{T^* \Bun^d}$.  
  The induced filtration on $\O(B^{(1)})\imbed A_d$ coincides with the one coming from the grading on $\O(B^{(1)})$ multiplied by $p$, and the associated graded map to this embedding is dual to $\mathrm{Fr}_B\circ h\colon \Hi \to B^{(1)}$.   It is known that all global functions on $T^*\Bun^d$ are pullbacks from the Hitchin base, so we have $\Gamma(\gr\DBud) = \O(B)$.  On the other hand, there is an inclusion $\gr A\imbed\Gamma(\gr\DBud)=\O(B)$.  Thus $\gr A$ identifies with an $\O(B^{(1)})$-submodule in $\O(B)$, therefore it is finitely generated and torsion-free over $\O(B^{(1)})$.  But then so is $A_d$
  , as desired.
\end{proof}

As explained above, the lemma shows that the map $\Pi$ is an isomorphism, so that for any $d$ we have a canonical isomorphism
$$A_d\isom A:=\O(\Op).$$
It also follows that $\Funiv^d$ are identified for all \dinZ, so we write $\Funiv$ for the sheaf isomorphic to all of them.

\begin{Lem}\label{c_comm_A}
  The map $c\colon \DBudC =\DBud \boxtimes \D_C \to H(\DBud[d-1])$ from Proposition~\ref{filtr} intertwines the two $A$-actions coming from the $A$-action on $\DBu$.
\end{Lem}

\begin{proof}
  Since  both the source and the target are torsion-free as modules over $\O(B^{(1)})\subset Z(\DBumod)$ (where $Z$ stands for the  center of a category, i.e., the ring of endomorphisms of the
  identity functor), it is enough to show that the statement of the lemma holds after tensoring by $\O(\Op_s)$ over $\O(\Op)$.  The localized map $c_s$ is a morphism in $(\DBuCmod)_s$, and we can apply $\Phi$ in the first factor.  Then Proposition \ref{Ops} and Hecke eigen-property of $\Phi$ imply that $\Phi(c)$ is a map of $(\O_{\Loc_s}\boxtimes \D_C)$-modules scheme-theoretically supported on $\Op_s/\Gm\x C \subset\Loc_s\x C$, and that the elements of $A=\O(\Op)$ act by multiplication by the same functions on this support.
\end{proof}

\subsection{End of proof of Theorem \ref{main} in the positive characteristic case}

Thus we proved the first part of the theorem.
The second statement of the theorem follows
from the construction of the morphism $\Pi$.
Indeed, we need to construct an isomorphism
\[H(\DBu)\cong \DBu \boxtimes_A \Funiv,
\]
where we used the following notation: if $A$ is a $k$-algebra, $X, Y$ are
$k$-schemes (or stacks) and $\F$, respectively  $\G$, are quasi-coherent sheaves on $X$,
resp.\ $Y$, with a right, resp.\ left, $A$-actions, then we define a quasi-coherent
sheaf on $X\x Y$: $\F\boxtimes_A \G:=(\F\boxtimes \G)\otimes_{A\otimes A^{op}}A$,
where the rightmost symbol $A$ refers to the regular $A$-bimodule.

We will construct the isomorphism as above for each component $\Bun^d$ of $\Bun$, so fix \dinZ.
Consider the functor $\DBud\boxtimes_A{-}\colon (A\otimes\D_C)\hmod\to \DBudCmod$, which is the left adjoint to the ``$d$'th component'' of the functor $\Upsilon$ used above.  So we have a counit map $a_M\colon \DBud\boxtimes_A\Upsilon(M)\to M$ for any $M\in\DBudCmod$.  We need to check that it is an isomorphism for $M=H(\DBud[d-1])$.  It is clear that if we apply the forgetful functor $\DBuCmod \to \Dmod_{\Bun^d\x C /C}$ to $a_M$ then we will get the counit morphism for the similar adjunction between $\Dmod_{\Bun^d\x C/C}$ and $(A\otimes\O_C)\hmod$.  Now since locally over~$C$, $H(\DBud[d-1])$ is isomorphic to $(\DBud\boxtimes\O_C)^{\oplus n}$ as a relative $\D$-module on $\Bun^d\x C$ over~$C$, and our statement is local in $C$, it suffices to check for $M=\DBud\boxtimes\O_C$ where it is clear.

Now, to deduce Hecke eigen-property, we have to show that the constructed isomorphism commutes with the $A$-action, where the action  on $H(\DBu)$ comes by transport of structure from the $A$-action on $\DBu$ and the action on $\DBu \boxtimes_A \Funiv$ comes from either of the $A$-actions contracted by the tensor product.  For this we consider the commutative diagram
\[\xymatrix@C=1.5em{
\DBud\bx\DC\ar@{}[r]|-{\dis=\!=}\ &\ \DBud\bx_A\Upsilon(\DBudC) \ar[rr]^-{a_{\DBudC}}_-\sim \ar[d]_{\DBud\bx_A\Upsilon(c)}\ &&\ \DBudC \ar[d]_c \\
\DBud\bx_A\Funiv\ar@{}[r]|-{\dis=\!=} \ &\ \DBud\bx_A\Upsilon(H(\DBud[d-1])) \ar[rr]^-{a_{H(\DBud[d-1])}}_-\sim \ &&\ H(\DBud[d-1]).\
}\]
By Lemma~\ref{c_comm_A}, the vertical arrows in this diagram commute with the $A$-action. The top arrow commutes with the $A$-action because of the naturality of $a_M$.  Hence the bottom arrow does, too, which is what we need.
Note that $\DBud\bx_A\Upsilon(\DBudC)$ a~priori has {\em three} $A$-actions, coming from the $A$-actions on $\DBud$, $\Ups$ and $\DBudC$. The first two actions are contracted by the tensor product and hence give rise to the same action, while the last two actions agree because under the isomorphism $\Ups(\DBudC)\isoto A\tsr\DC$ they go to the right, resp.\ left, actions on the first factor, and since $A$ is commutative, they are the same. Note that the first action is adapted to showing that the left vertical arrow commutes with $A$, and the last one is adapted to the horizontal map.

What we just proved implies the following in-families version of Hecke eigen-property.  The $A$-action on the $\D$-module $\DBu$ allows to present it as a pushforward of a $(\Spec A=\Op)$-family $\cM$ of $\D$-modules.  Applying to~$\cM$ the relative version of the Hecke functor~$H$, we get a relative (twisted) $\D$-module on $\Bun\x\Op\x C$ over $\Op$ which we denote by $H_{\Op}(\cM)$.  Then it follows from what we proved that $H_{\Op}(\cM) \cong (\cM \bx \O_C) \tsr (\O_{\Bun}\bx \Funiv')$ where $\Funiv'$ is the $\Op$-family of opers constructed from $\Funiv$.  Taking pullback to a closed point $x$ of $\Op$, we see that the derived specialization of the family~$\cM$, $\cM_x:= L(\{x\}\x\Bun \to \Op\x\Bun)^*\cM = \DBu\Ltsr_A k_x$, is  a Hecke eigen-$\D$-module.

We want to show that $\cM_x$ is actually in the heart of $D(\Bun)$.  In other words, we want to prove that $\DBu$ is flat over~$A$.
Since the $A$-module  $\DBu$ is a (flat) deformation of $\O_{\Hi}$ viewed as a module
over $\O(B)$, this follows from flatness of the Hitchin map.
\qed


\section{Proof of the main theorem in the case $\chark=0$}

Considering the deformation of $\O_{T^*\Bun^d}$ to $\DBud$ we get an injection 
\[ \gr A_d \into \Gamma(\O_{T^*\Bun^d})=\Gamma(\O_B) .\]

Choose a finitely generated (over $\ZZ$) subring $R\subset k$ and a smooth proper
curve $C_R$ over $R$ whose base change to $k$ is isomorphic to $C$.
We have a similar injection constructed from the moduli stack $\Bun_R$ for vector
bundles over $C_R$:
\[ \gr A_{d,R}\into  \Gamma(\O_{T^*\Bun^d_R})=\Gamma(\O_{B_R}) .\]
Here we use that the relative $D$-module $D_{Bun_R/R}$ is flat over $R$.  This follows from the fact that $\underline{Bun}_R$ is a family of good stacks in the sense of 
\cite[\S 1.1.1]{BD}: over a characteristic zero field this is proved  in \cite[Prop.2.1.2]{BD}, the general case is similar.

Since the right hand side is finitely generated  as an $R$-module in every degree, equality of its elements can be checked by reduction modulo all maximal ideals of $R$.  The same property therefore holds for $\gr A_{d,R}$ and hence for $A_{d,R}$ itself.  For every such maximal ideal $\fm$, the residue field $\bbF :=R/\fm$ is finite and  we have an injective map
$\iota_{d,R}\colon   A_{d,R}\otimes_R \bbF \into A_{d,\bbF}$ whose target is commutative by the previous section.  Hence so is the source and, since this holds for all~$\fm$, we get that $A_d$ is commutative too.

Now  
the construction of the previous section
yields a family of opers parametrized by $\Spec(A_d)$,
given by $\Ups(H(\DBud[d-1]))$ as before.
 Thus we get a map
$\Pi\colon \Spec(A_d) \to \Op$ as explained in the previous section.
Unraveling the definition, one can see that $\Pi=\Pi_R\tsr_Rk$ and $  \Pi_{\bbF} ^* =\iota_{d,R}\circ(\Pi_R^*\tsr_R\bbF)$.
But $\Pi_\bbF$ is an isomorphism by the previous section, and  $\iota_{d,R}$ is injective, hence both
 $\iota_{d,R}$ and $\Pi_R^*\tsr_R\bbF$ are actually isomorphisms.

\quash{
Considering the deformation of $\O_{T^*\Bun^d}$ to $\DBud$ we get a spectral sequence
$$\Gamma(\O_{T^*\Bun^d})=\Gamma(\O_B) \Rightarrow A_d.$$

Choose a finitely generated ring $R$ with a homomorphism  $R\to k$ and a complete
curve $C_R$ over $R$ whose base change to $k$ is isomorphic to $C$.
We can form a similar spectral sequence starting from the moduli stack $\Bun_R$ for vector
bundles over $C_R$.
Its base change to a field of positive characteristic degenerates at $E_1$ by the result
of the previous section. Hence the spectral sequence itself, as well as its base change to $k$, degenerates at $E_1$.

This implies that $A_d$ is commutative for any $d$ , since its base change to any residue field of~$R$ of  almost every prime characteristic is commutative.   \todo[Do we really need the degeneration of the spectral sequence for that?]
Now  
the construction of the previous section
yields a family of opers parametrized by $\Spec(A_d)$,
given by $\Ups(H(\DBud[d-1]))$ as before.
 Thus we get a map
$\Pi\colon \Spec(A_d) \to \Op$ as explained in the previous section.

Since the base change of $\Pi$ to a field of almost any prime characteristic is an isomorphism,
}

Since the base change of $\Pi_R$ to all finite residue fields of~$R$ are isomorphisms,
we see that $\Pi_R$, and hence $\Pi=\Pi_R\tsr_Rk$, is an isomorphism. This proves the first part of Theorem \ref{main}.
One then proves an analogue of Lemma~\ref{c_comm_A} in characteristic~$0$ by observing that it is enough to prove the statement for the reductions to finite residue fields of~$R$.
After that,
the second part follows 
by the argument of the previous section.

\numberwithin{thm}{section}
\newtheorem{corollary}[thm]{Corollary}
\newtheorem{lemma}[thm]{Lemma}
\newtheorem{proposition}[thm]{Proposition}
\newtheorem{conjecture}[thm]{Conjecture}
\newtheorem{thm-dfn}[thm]{Theorem-Definition}

\newtheorem*{coro}{Corollary} 

\theoremstyle{definition}
\newtheorem{definition}[thm]{Definition}
\newtheorem{remark}[thm]{Remark}
\newtheorem{example}[thm]{Example}

\numberwithin{equation}{section}

\providecommand{\cC}{{\mathcal C}}
\providecommand{\fg}{{\mathfrak g}}
\providecommand{\ft}{{\mathfrak t}}
\providecommand{\fl}{{\mathfrak l}}
\providecommand{\fb}{{\mathfrak b}}
\providecommand{\fu}{{\mathfrak u}}
\providecommand{\fp}{{\mathfrak p}}
\providecommand{\fz}{{\mathfrak z}}
\providecommand{\ff}{{\mathfrak F}}
\providecommand{\fa}{{\mathfrak a}}
\providecommand{\fC}{{\mathfrak C}}
\providecommand{\fX}{{\mathfrak X}}
\providecommand{\fc}{{\mathfrak c}}

\renewcommand{\rZ}{{\mathrm Z}}
\renewcommand{\rW}{{\mathrm W}}
\renewcommand{\rU}{{\mathrm U}}
\newcommand{\rUg}{\mathrm U(\fg)}
\newcommand{\rUl}{\mathrm U(\fl)}

\providecommand{\hL}{\widehat{\lambda+\Lambda}}
\providecommand{\ol}{\overline{\lambda}}
\providecommand{\hl}{\hat{\lambda}}
\providecommand{\hol}{\widehat{\lambda+\Lambda}}

\providecommand{\bC}{{\mathbb C}}
\providecommand{\bX}{{\mathbb X}}
\providecommand{\bG}{{\mathbb G}}
\providecommand{\bZ}{{\mathbb Z}}

\providecommand{\tD}{\widetilde{\D}}

\providecommand{\mD}{\mathcal{D}}
\providecommand{\mS}{\mathcal{S}}
\providecommand{\mI}{\mathcal{I}}
\providecommand{\mY}{\mathcal{Y}}
\providecommand{\mX}{\mathcal{X}}
\providecommand{\mE}{\mathcal{E}}
\providecommand{\mF}{\mathcal{F}}
\providecommand{\mQ}{\mathcal{Q}}
\providecommand{\mJ}{J}
\providecommand{\mA}{\mathcal{A}}
\providecommand{\mM}{\mathcal{M}}
\providecommand{\mT}{\mathcal{T}}
\providecommand{\mO}{\mathcal{O}}
\providecommand{\mL}{\mathcal{L}}
\providecommand{\mH}{\mathcal{H}}
\providecommand{\mP}{\mathcal{P}}
\providecommand{\mad }{\mathcal{A}^{\diamond}}
\providecommand{\mG}{\mathcal{G}}
\providecommand{\mN}{\mathcal{N}}
\providecommand{\mB}{\mathcal{B}}

\providecommand{\sX}{\mathscr{X}}
\providecommand{\sY}{\mathscr{Y}}
\providecommand{\sS}{\mathscr{S}}
\providecommand{\sG}{\mathscr{G}}
\providecommand{\sB}{\mathscr{B}}
\providecommand{\sP}{\mathscr{P}}
\providecommand{\sT}{\mathscr{T}}
\providecommand{\sA}{\mathscr{A}}
\providecommand{\sM}{\mathscr{M}}
\providecommand{\scP}{\mathscr{P}}
\providecommand{\sH}{\mathscr{H}}
\providecommand{\sL}{\mathscr{L}}
\providecommand{\sQ}{\mathscr{Q}}
\providecommand{\sC}{\mathscr{C}}
\providecommand{\sD}{\mathscr{D}}

\providecommand{\on}{\operatorname}
\providecommand{\D}{\on{D}}
\providecommand{\h}{\on{H}}
\providecommand{\tl}{t^{\lambda}}
\providecommand{\tU}{\widetilde U}

\providecommand{\oP}{\overline{P}}
\providecommand{\oN}{\overline{N}}
\providecommand{\uP}{\underline{P}}

\providecommand{\ra}{\rightarrow}
\providecommand{\la}{\leftarrow}
\providecommand{\gl}{\geqslant}
\providecommand{\s}{\textbackslash}
\providecommand{\bs}{\backslash}
\providecommand{\is}{\simeq}

\providecommand{\Loc}{\on{LocSys}}
\providecommand{\Bun}{\on{Bun}}
\providecommand{\Sect}{\on{Sect}}
\providecommand{\tC}{\widetilde C}

\providecommand{\nc}{\providecommand}

\providecommand{\fraka}{{\mathfrak a}}
\providecommand{\frakb}{{\mathfrak b}}
\providecommand{\frakc}{{\mathfrak c}}
\providecommand{\frakd}{{\mathfrak d}}
\providecommand{\frake}{{\mathfrak e}}
\providecommand{\frakf}{{\mathfrak f}}
\providecommand{\frakg}{{\mathfrak g}}
\providecommand{\frakh}{{\mathfrak h}}
\providecommand{\fraki}{{\mathfrak i}}
\providecommand{\frakj}{{\mathfrak j}}
\providecommand{\frakk}{{\mathfrak k}}
\providecommand{\frakl}{{\mathfrak l}}
\providecommand{\frakm}{{\mathfrak m}}
\providecommand{\frakn}{{\mathfrak n}}
\providecommand{\frako}{{\mathfrak o}}
\providecommand{\frakp}{{\mathfrak p}}
\providecommand{\frakq}{{\mathfrak q}}
\providecommand{\frakr}{{\mathfrak r}}
\providecommand{\fraks}{{\mathfrak s}}
\providecommand{\frakt}{{\mathfrak t}}
\providecommand{\fraku}{{\mathfrak u}}
\providecommand{\frakv}{{\mathfrak v}}
\providecommand{\frakw}{{\mathfrak w}}
\providecommand{\frakx}{{\mathfrak x}}
\providecommand{\fraky}{{\mathfrak y}}
\providecommand{\frakz}{{\mathfrak z}}
\providecommand{\frakA}{{\mathfrak A}}
\providecommand{\frakB}{{\mathfrak B}}
\providecommand{\frakC}{{\mathfrak C}}
\providecommand{\frakD}{{\mathfrak D}}
\providecommand{\frakE}{{\mathfrak E}}
\providecommand{\frakF}{{\mathfrak F}}
\providecommand{\frakG}{{\mathfrak G}}
\providecommand{\frakH}{{\mathfrak H}}
\providecommand{\frakI}{{\mathfrak I}}
\providecommand{\frakJ}{{\mathfrak J}}
\providecommand{\frakK}{{\mathfrak K}}
\providecommand{\frakL}{{\mathfrak L}}
\providecommand{\frakM}{{\mathfrak M}}
\providecommand{\frakN}{{\mathfrak N}}
\providecommand{\frakO}{{\mathfrak O}}
\providecommand{\frakP}{{\mathfrak P}}
\providecommand{\frakQ}{{\mathfrak Q}}
\providecommand{\frakR}{{\mathfrak R}}
\providecommand{\frakS}{{\mathfrak S}}
\providecommand{\frakT}{{\mathfrak T}}
\providecommand{\frakU}{{\mathfrak U}}
\providecommand{\frakV}{{\mathfrak V}}
\providecommand{\frakW}{{\mathfrak W}}
\providecommand{\frakX}{{\mathfrak X}}
\providecommand{\frakY}{{\mathfrak Y}}
\providecommand{\frakZ}{{\mathfrak Z}}
\providecommand{\bbA}{{\mathbb A}}
\providecommand{\bbB}{{\mathbb B}}
\providecommand{\bbC}{{\mathbb C}}
\providecommand{\bbD}{{\mathbb D}}
\providecommand{\bbE}{{\mathbb E}}
\providecommand{\bbF}{{\mathbb F}}
\providecommand{\bbG}{{\mathbb G}}
\providecommand{\bbH}{{\mathbb H}}
\providecommand{\bbI}{{\mathbb I}}
\providecommand{\bbJ}{{\mathbb J}}
\providecommand{\bbK}{{\mathbb K}}
\providecommand{\bbL}{{\mathbb L}}
\providecommand{\bbM}{{\mathbb M}}
\providecommand{\bbN}{{\mathbb N}}
\providecommand{\bbO}{{\mathbb O}}
\providecommand{\bbP}{{\mathbb P}}
\providecommand{\bbQ}{{\mathbb Q}}
\providecommand{\bbR}{{\mathbb R}}
\providecommand{\bbS}{{\mathbb S}}
\providecommand{\bbT}{{\mathbb T}}
\providecommand{\bbU}{{\mathbb U}}
\providecommand{\bbV}{{\mathbb V}}
\providecommand{\bbW}{{\mathbb W}}
\providecommand{\bbX}{{\mathbb X}}
\providecommand{\bbY}{{\mathbb Y}}
\providecommand{\bbZ}{{\mathbb Z}}
\providecommand{\calA}{{\mathcal A}}
\providecommand{\calB}{{\mathcal B}}
\providecommand{\calC}{{\mathcal C}}
\providecommand{\calD}{{\mathcal D}}
\providecommand{\calE}{{\mathcal E}}
\providecommand{\calF}{{\mathcal F}}
\providecommand{\calG}{{\mathcal G}}
\providecommand{\calH}{{\mathcal H}}
\providecommand{\calI}{{\mathcal I}}
\providecommand{\calJ}{{\mathcal J}}
\providecommand{\calK}{{\mathcal K}}
\providecommand{\calL}{{\mathcal L}}
\providecommand{\calM}{{\mathcal M}}
\providecommand{\calN}{{\mathcal N}}
\providecommand{\calO}{{\mathcal O}}
\providecommand{\calP}{{\mathcal P}}
\providecommand{\calQ}{{\mathcal Q}}
\providecommand{\calR}{{\mathcal R}}
\providecommand{\calS}{{\mathcal S}}
\providecommand{\calT}{{\mathcal T}}
\providecommand{\calU}{{\mathcal U}}
\providecommand{\calV}{{\mathcal V}}
\providecommand{\calW}{{\mathcal W}}
\providecommand{\calX}{{\mathcal X}}
\providecommand{\calY}{{\mathcal Y}}
\providecommand{\calZ}{{\mathcal Z}}

\nc{\al}{{\alpha}} \nc{\be}{{\beta}} \nc{\ga}{{\gamma}}
\nc{\ve}{{\varepsilon}} \nc{\Ga}{{\Gamma}} 
\nc{\La}{{\Lambda}}

\nc{\ad }{{\on{ad }}}

\nc{\aff}{{\on{aff}}} \nc{\Aff}{{\mathbf{Aff}}}
\providecommand{\Aut}{{\on{Aut}}}
\providecommand{\cha}{{\on{char}}}
\nc{\der}{{\on{der}}}
\providecommand{\Der}{{\on{Der}}}
\nc{\diag}{{\on{diag}}}
\providecommand{\End}{{\on{End}}}
\nc{\Fl}{{\calF\ell}}
\providecommand{\Gal}{{\on{Gal}}}
\providecommand{\Gr}{{\on{Gr}}}
\nc{\Hg}{{\on{Higgs}}}
\providecommand{\Hom}{{\on{Hom}}}
\providecommand{\id}{{\on{id}}}
\nc{\Id}{{\on{Id}}}
\providecommand{\ind}{{\on{ind}}}
\nc{\Ind}{{\on{Ind}}}
\providecommand{\Lie}{{\on{Lie}}}
\nc{\Op}{{\on{Op}}}
\providecommand{\Pic}{{\on{Pic}}}
\providecommand{\pr}{{\on{pr}}}
\providecommand{\Res}{{\on{Res}}}
\nc{\res}{{\on{res}}}
\providecommand{\sgn}{{\on{sgn}}}
\providecommand{\Spec}{{\on{Spec}}}
\providecommand{\St}{{\on{St}}}
\nc{\tr}{{\on{tr}}}
\providecommand{\Tr}{{\on{Tr}}}
\providecommand{\Mod}{{\mathrm{-Mod}}}
\providecommand{\GL}{{\on{GL}}}
\nc{\GSp}{{\on{GSp}}} \nc{\GU}{{\on{GU}}} \nc{\SL}{{\on{SL}}}
\nc{\SU}{{\on{SU}}} \nc{\SO}{{\on{SO}}}

\nc{\nh}{{\Loc_{J^p}(\tau')}}
\nc{\bnh}{{\Loc_{\breve J^p}(\tau')}}
\providecommand{\bGr}{{\overline{\Gr}}}
\nc{\bU}{{\overline{U}}} \nc{\IC}{{\on{IC}}}
\providecommand{\Nm}{{\on{Nm}}}
\providecommand{\ppart}{(\!(t)\!)}
\providecommand{\pparu}{(\!(u)\!)}

\def\xcoch{\mathbb{X}_\bullet}
\def\xch{\mathbb{X}^\bullet}

\providecommand{\AJ}{{\on{AJ}}}

\begin{appendix}
\section{Hitchin map and opers in characteristic $p$ (by Roman Bezrukavnikov, Tsao-Hsien Chen, and Xinwen Zhu)}

This appendix is devoted to the proof of the following statement (Proposition~\ref{finite_flat} of the main text). 
Let $G$ be a connected reductive algebraic group over $k$.
We assume that the characteristic $\chark=p$ does not divide the order of the Weyl group $\rW$ of $G$.

\begin{thm} 
\label{a main}
Let $\on{Op}_{G}$ be the scheme of $G$-opers with marking
(see \S\ref{oper}).
Then the composition
\[\pi_p:\on{Op}_{G}\ra\Loc_G\xrightarrow{h_p} B^{(1)}\]
is finite and faithfully flat of degree $p^{\dim B}$.
Here $h_p$ is the $p$-Hitchin map.
\end{thm}

\begin{remark}
In the case $G=PGL_n$,
the theorem above is a strengthening
of a result of C. Pauly and K. Joshi \cite{JP} who proved that the $p$-Hitchin map on the space of opers is finite.
\end{remark}

\subsection{Notations}
Let $C$ be a complete smooth curve over $k$.
Let $G$ be a connected reductive algebraic group over $k$ of rank $l$.
We denote by
$\fg$ the Lie algebras of $G$. We fix a
Borel subgroup $B_G\subset G$, and let $N$ be its unipotent radical and $T=B_G/N$.
Let $Z(G)$ be the center of $G$. We denote by $G_{ad}=G/Z(G)$, $B_{ad}=B_G/Z(G)$
and $T_{ad}=T/Z(G)$.
We denote the
corresponding Lie algebras by  $\fb$, $\frakn$ and $\ft$. 

\subsection{Hitchin
map and $p$-Hitchin map}
In this subsection, we recall the definition of Hitchin and $p$-Hitchin map
following \cite{N,CZ1,BB}.
\subsubsection{Hitchin map}
Let $k[\fg]$ and $k[\ft]$ be the algebras of polynomial functions on
$\fg$ and $\ft$. By Chevalley's theorem, we have an isomorphism
$k[\fg]^{G}\is k[\ft]^{\rW}$. Moreover, $k[\ft]^\rW$ is isomorphic to
a polynomial ring of $l$ variables $u_1,\ldots,u_l$ and each $u_i$ is
homogeneous in degree $e_i$. Let $\fc=\on{Spec}(k[\ft]^W)$. Let  $$\chi:\fg\ra\fc$$ be the map induced by $k[\fc]\is
k[\fg]^G\hookrightarrow k[\fg]$.
This is $G\x\bG_m$-equivariant
map where $G$ acts trivially on $\fc$, and $\bG_m$ acts on $\fc$ through the gradings on $k[\ft]^{\rW}$. Let $\mL$ be an invertible
sheaf on $C$ and $\mL^\x$ be the corresponding
$\bG_m$-torsor. Let
$\fg_{\mL}=\fg\x^{\bG_{m}}\mL^\x$ and
$\fc_{\mL}=\fc\x^{\bG_m}\mL^\x$ be  the $\bG_m$-twist of $\fg$
and $\fc$ with respect to the natural $\bG_m$-action.

Let $\on{Higgs}_{G,\mL}=\on{Sect}(C,[\fg_{\mL}/G])$ be the stack of
section of $[\fg_{\mL}/G]$ over $C$, i.e., for each $k$-scheme $S$
the groupoid  $\on{Higgs}_{G,\mL}(S)$ consists of maps over $C$:
\[h_{E,\phi}:C\x S\ra[\fg_{\mL}/G].\]
Equivalently, $\on{Higgs}_{G,\mL}(S)$ consists of a pair $(E,\phi)$ (called a Higgs bundle), where
$E$ is a $G$-torsor over $C\x S$ and $\phi$ is an element in
$\Gamma(C\x S,\ad (E)\otimes \mL)$. If the group $G$ is clear
from the content, we simply write $\on{Higgs}_{\mL}$ for
$\on{Higgs}_{G,\mL}$.

Let $B_{\mL}=\on{Sect}(C,\fc_{\mL})$ be the scheme of
sections of $\fc_{\mL}$ over $C$, i.e., for each $k$-scheme $S$,
$B_{\mL}(S)$ is the set of sections over $C$
$$b:C\x S\ra\fc_{\mL}.$$
This is called the Hitchin base of $G$.

The natural $G$-invariant projection $\chi:\fg\ra\fc$  induces a map
$$[\chi_{\mL}]:[\fg_{\mL}/G]\ra\fc_{\mL},$$
which in turn induces a natural map
$$h_{\mL}:\on{Higgs}_{\mL}=\on{Sect}(C,[\fg_{\mL}/G])\ra\on{Sect}(C,\fc_{\mL})=B_{\mL}.$$
We call $h_{\mL}:\on{Higgs}_{\mL}\ra B_{\mL}$ the Hitchin map
associated to $\mL$.

We are mostly interested in the case $\mL=\omega$. For
simplicity, from now on we denote $B=B_{\omega}$,
$\on{Higgs}=\on{Higgs}_{\omega}$ and $h=h_{\omega}$, etc. We sometimes also write
$\on{Higgs}_G$ for $\on{Higgs}$ to emphasize the group $G$.

We fix a
square root $\kappa=\omega^{1/2}$ (called a theta characteristic of
$C$). Recall that in this case, there is a section
$\epsilon_{\kappa}:B\ra\on{Higgs}$ of $h:\Hg\to B$, induced by the Kostant section $kos:\fc\ra\fg$. Sometimes, we also call $\epsilon_{\kappa}$ the Kostant section of the Hitchin fibration.
\quash{
\subsubsection{Kostant section}\label{kostant section}
In this section, we recall the construction of the Kostant section
of Hitchin map $h_{\mL}$. For each simple root $\alpha_i$ we choose
a nonzero vector $f_i\in\fg_{-\alpha_i}$. Let $f=\oplus_{i=1}^l
f_i\in\fg$.  We complete $f$  into a $sl_2$ triple $\{f,h,e\}$ and
we denote by $\fg^e$ the centralizer of $e$ in $\fg$. A theorem of
Kostant says that $f+\fg^e$ consists of regular element in $\fg$ and
the restriction of $\chi:\fg\ra\fc$ to $f+\fg^e$ is an isomorphism
onto $\fc$. We denote by $kos:\fc\is f+\fg^e$ to be the inverse of
$\chi|_{f+\fg^e}$. Let $\rho(\bG_m)$ be the $\bG_m$ action on $\fg$
described below. It acts trivially on $\ft$ and on $\fg_{\alpha}$
the action is given by $\rho(t)x=t^{\on{ht}(\alpha)}x$ where
$\on{ht}(\alpha)=\sum n_i$ if $\alpha=\sum n_i\alpha_i$. We have
$\rho(t)f=t^{-1}f$ and $\rho(t)e=te$, in particular $\fg^e$ is
invariant under $\rho(\bG_m)$. We define a new $\bG_m$-action on
$\fg$ by  $\rho^+(t)=t\rho(t)$. We have $\rho^+(t)f=f$ and
$\rho^+(\bG_m)$ preverse $f+\fg^e$. The isomorphism $kos:\fc\is
f+\fg^e$ is $\bG_m$-equivariant where $ f+\fg^e$ is acted by
$\rho^+(\bG_m)$.

For any $k$ scheme $S$, the groupoid $\on{Higgs}_{\mL}(S)$ consists of
maps
\[h_{E,\phi}:S\x C\ra [\fg/G\x\bG_m]\] such that the composition of
$h_{E,\phi}$ with the projection $[\fg/G\x\bG_m]\ra B\bG_m$ is
given by the $\bG_m$-torsor $\rho_L$. Similarly, $B_L(S)$ can be
regarded as maps
\[b:S\x C\ra [\fc/\bG_m]\]
such that the composition of $b$ with the projection $[\fc/\bG_m]\ra
B\bG_m$ is given by $\mL^\x$. Clearly, the Hithch map $h_{\mL}$
is induced by the natural map $[\chi/G\x
\bG_m]:[\fg/G\x\bG_m]\ra [\fc/\bG_m]$.

The diagonal map $\bG_m\ra\bG_m\x\bG_m$ induced a map
\[[\fg/\rho^+(\bG_m)]\ra[\fg/\bG_m\x\rho(\bG_m)].\]
By precomposing with the map
$[\fc/\bG_m]\stackrel{kos}\is[f+\fg^e/\rho^+(\bG_m)]\ra[\fg/\rho^+(\bG_m)]$
we obtain
\[[\fc/\bG_m]\ra[\fg/\bG_m\x\rho(\bG_m)].\]

If the acton of $\rho(\bG_m)$ on $\fg$  factors through the adjoint
action of $G$, for example when $G$ is adjoint, then there is a map
$[\fg/\bG_m\x\rho(\bG_m)]\ra[\fg/\bG_m\x G]$ and it defines
a section
\[[\fc/\bG_m]\ra[\fg/\bG_m\x\rho(\bG_m)]\ra[\fg/\bG_m\x G]\]
of $[\chi/G\x\bG_m]$, in particular, we get a section of
$h_{\mL}$. In the general case, the action $\rho(\bG_m)$ does not factor
through $G$, but its square does and it is given by the co-character
$2\rho:\bG_m\ra G$ where $2\rho$ is the sum of positive coroots. So
if we consider the square map $\bG_m^{[2]}\ra\bG_m$, we get a map
\[\eta^{1/2}:[\fc/\bG_m^{[2]}]\ra[\fg/\bG_m^{[2]}\x\rho(\bG_m^{[2]})]\ra[\fg/\bG_m^{[2]}\x G].\]

Let $\mL^{1/2}$ be a square root of  $\mL$. For any $b:S\x
C\ra[\fc/\bG_m]$ in $B_{\mL}(S)$, it factors through a unique map
$b^{1/2}:S\x C\ra[\fc/\bG_m^{[2]}]$. Therefore, by composing
with $\eta^{1/2}$, we get a lift of $b$:
\[\eta^{1/2}(b):S\x C\stackrel{b^{1/2}}\longrightarrow[\fc/\bG_m^{[2]}]\stackrel{\eta^{1/2}}
\longrightarrow[\fg/\bG_m^{[2]}\x G]\ra[\fg/\bG_m\x G].\]
The assignment $b\ra\eta^{1/2}(b)$ defines a section
\[\eta_{\mL^{1/2}}:B_{\mL}\ra\on{Higgs}_{\mL}\]
of Hitchin map $h_{\mL}$.

We fix a
square root $\kappa=\omega^{1/2}$ (called a theta characteristic) of
$\omega$ and write
$\kappa=\eta_{\kappa}:B\ra\on{Higgs}$.

}

\quash{
\subsection{The universal centralizer group schemes}\label{univ cent}
Consider the group scheme $I$ over $\fg$ consisting of pair
$$I=\{(g,x)\in G\x\fg\mid \on{Ad}_g(x)=x\}.$$
We define $J=kos^{*}I$, This is called the universal centralizer group scheme of
$\fg$, which is a smooth commutative group scheme over $\fc$.
The following proposition is proved in \cite{DG, N1}
\begin{proposition}\label{J}
There is a canonical isomorphism of group schemes
$\chi^{*}J|_{\fg^{reg}}\is I|_{\fg^{reg}}$, which extends to a
morphism of group schemes $a:\chi^{*}J\ra I\subset G\x\fg$.
\end{proposition}

All the above constructions can be twisted. Namely, there are
$\bG_m$-actions on $I$, $J$. Moreover, the
$\bG_m$-action on $I$ can be extended to a $G\x\bG_m$-action
given by $(h,t)\cdot(x,g)=(tx,hgh^{-1})$. The natural morphisms $J
\ra\fc$ and $I\ra \fg$ are $\bG_m$-equivariant, therefore we can
twist everything by the $\bG_m$-torsor $\mL^\x$ and get
$J_{\mL}\ra\fc_{\mL}$, $I_{\mL}\ra\fg_{\mL}$ where
$J_{\mL}=J\x^{\bG_m}\mL^\x$ and
$I_{\mL}=I\x^{\bG_m}\mL^\x$. The group
scheme $I_{\mL}$ over $\fg_{\mL}$ is equivariant  under the
$G$-action, hence it descends to a group scheme $[I_{\mL}]$ over
$[\fg_{\mL}/G]$. For simplicity, $J=J_{\omega}$ and $I=I_{\omega}$ if no
confusion will arise.


\subsection{Symmetries of Hitchin fibration}\label{symmetry of Hitchin}
Let $b:S\ra B_{\mL}$ be $S$-point of $B_{\mL}$. It corresponds to a map
$b:C\x S\ra\fc_{\mL}$. Pulling back $J\ra\fc_{\mL}$
using $b$ we get a smooth groups scheme $J_b=b^{*}J$ over $C\x
S$.

Let $\sP_b$ be the Picard category of $J_b$-torsors over $C\x
S$. The assignment $b\ra\sP_b$ defines a Picard stack over $B$,
denoted by $\sP_\mL$. Let $b\in B_\mL(S)$.  Let $(E,\phi)\in\on{Higgs}_{\mL,b}$
and let $h_{E,\phi}:C\x S\ra [\fg_{\mL}/G]$ be the
corresponding map. Observe that the morphism $\chi^*J\to I$ in
Proposition \ref{J} induces $[\chi_{\mL}]^{*}J\ra [I]$ of group
schemes over $[\fg_{\mL}/G]$.
 Pulling back to $C\x S$ using $h_{E,\phi}$, we get a map

\begin{equation}\label{actEphi}
a_{E,\phi}:J_b\ra h_{E,\phi}^{*}[I]=\Aut(E,\phi)\subset \Aut(E).
\end{equation}
Therefore, using the  map $a_{E,\phi}$ we can twist
$(E,\phi)\in\on{Higgs}_{\mL,b}$ by a $J_b$-torsor. This defines an action
of $\sP_{\mL}$ on $\on{Higgs}_{\mL}$ over $B_{\mL}$.

Let $\on{Higgs}^{reg}_{\mL}$ be the open stack of $\on{Higgs}_{\mL}$ consists of
$(E,\phi):C\to [\fg_\mL/G]$ that factors through $C\to
[(\fg^{reg})_\mL/G]$. If $(E,\phi)\in\on{Higgs}^{reg}_{\mL}$, then
$a_{E,\phi}$ above is an isomorphism. The Kostant section $\epsilon_{\mL^{1/2}}:
B_{\mL}\to \on{Higgs}_{\mL}$ factors through $\epsilon_{\mL^{1/2}}: B_{\mL}\to \on{Higgs}^{reg}_{\mL}$.  The following
proposition can be extracted from \cite{DG,N1}:

\begin{proposition}\label{reg torsor}
The stack $\on{Higgs}^{reg}_{\mL}$ is a $\sP_{\mL}$-torsor, which is trivialized
by a choice of Kostant section $\epsilon_{\mL^{1/2}}$.
\end{proposition}

\subsection{The tautological section $\tau:\fc\to \Lie J$}\label{tau and c}
Recall that by
Proposition \ref{J}, there is a canonical isomorphism
$\chi^*J|_{\fg^{reg}}\simeq I|_{\fg^{reg}}$. The sheaf of Lie
algebras $\Lie (I|_{\fg^{reg}})\subset \fg^{reg}\x \fg$ admits a
tautological section given by $x\mapsto x\in I_x$ for $x\in
\fg^{reg}$. Clearly, this section descends to give a tautological
section $\tau: \fc\to \Lie J$. We have the following property of $\tau$.

\begin{lem}\label{taut:prolong}
Let $x\in \fg$, and $a_x: J_{\chi(x)}\to I_x\subset G$ be the homomorphism as in Proposition \ref{J} (1). Then $da_x(\tau(x))=x$.
\end{lem}
\begin{proof}Consider the universal situation $x=\id:\fg\to\fg$. Then we need to show that $da\circ\tau:\fg\to  \chi^*\Lie J\to \fg\x \fg$ is the diagonal map. But by definition, this is true when restricted to $\fg^{reg}\subset \fg$. Therefore it holds over $\fg$.
\end{proof}
\begin{remark}In particular, if we take $x=0$, this shows that $da_0:\Lie J_{\chi(0)}\to \frakg$ is not injective. I.e., ,the map $a:\chi^*J\to I$ is either injective nor surjective over a general $x\in \fg$.
\end{remark}

Observe that $\bG_m$ acts on $\fg^{reg}\x \fg$ via natural
homothethies on both factors, and therefore on $\chi^*\Lie
J|_{\fg^{reg}}\simeq \Lie (I|_{\fg^{reg}})\subset \fg^{reg}\x
\fg$. This $\bG_m$-action on $\chi^*\Lie J|_{\fg^{reg}}$ descends to
a $\bG_m$-action on $\Lie J$ and for any line bundle $\mL$ on $C$,
the $\mL^\x$-twist $(\Lie J)\x^{\bG_m}\mL^\x$ under this
$\bG_m$ action is $\Lie (J_\mL)\otimes \mL$, where $J_\mL$ is
introduced in \S \ref{univ cent}. In addition, $\tau$ is $\bG_m$-equivariant with respect to this $\bG_m$ action on
$\Lie J$ and the natural $\bG_m$ action on $\fc$. Therefore, if we define a vector bundle
$B_{J,\mL}$ over $B_{\mL}$, whose fiber over $b\in B_{\mL}$ is
$\Gamma(C,\Lie J_b\otimes \mL)$, then by
twisting $\tau$ by $\mL$, we obtain
\begin{equation}\label{taut sect II}
\tau_{\mL}: B_{\mL}\to B_{J,\mL}.
\end{equation}
which is a canonical section of the  projection $\pr:B_{J,\mL}\to
B_{\mL}$.
As before, we omit the subscript ${_\mL}$ if $\mL=\omega$ for brevity.
}

\subsubsection{$p$-Hitchin map}
Let $\Loc_{G}$ be the stack of $G$-local systems on $C$, i.e., for
every scheme $S$ over $k$, $\Loc_{G}(S)$ is the groupoid of all
$G$-torsors $E$ on $C\x S$ together with a connection $\nabla:
T_{C\x S/S}\ra\widetilde{T}_E$, here $\widetilde{T}_E$ is the
Lie algebroid of infinitesimal symmetry of $E$.
Recall the notion of $p$-curvature of a $G$-local system following \cite{K,Bo}:
For any $(E,\nabla)\in\Loc_{G}$
the $p$-curvature  of $\nabla$ is defined as \[\Psi(\nabla): F^*T_{C'}\ra \ad(E), \quad v\ra\nabla(v)^p-\nabla(v^p).\]
We regard $\Psi(\nabla)$ as
an element
$\Psi(\nabla)\in\Gamma(C,\ad (E)\otimes\omega^p)$ and call such a pair an $F$-Higgs field.
The
assignment $(E,\nabla)\ra (E,\Psi(\nabla))$ defines a map
$\Psi_{G}:\Loc_{ G}\ra \on{Higgs}_{G,\omega^p}$.
Combining this map with $h_{\omega^p}$, we get a morphism from
$\Loc_{G}$ to $B_{\omega^p}$:
$$\tilde h_p:\Loc_{ G}\ra B_{\omega^p}.$$

Observe that the pullback along $F_C:C\to C^{(1)}$ induces a natural map
$F^p: B^{(1)}\ra B_{\omega^p}$, where the superscript denotes the
Frobenius twist.
By \cite[Theorem 3.1]{CZ1}, the $p$-curvature morphism $\tilde h_p:\Loc_{G}\ra B_{\omega^p}$
factors through a unique morphism $$h_p:\Loc_{ G}\ra B^{(1)}.$$ We called
this map the $p$-Hitchin map.


The construction of $p$-Hitchin map can be generalized to $\lambda$-connections.
Recall that for any $\lambda\in k$, a $\lambda$-connection on a $G$-torsor $E$
is an $\mO_C$-linear map
$\nabla_\lambda:T_C\ra\widetilde T_E$
such that the composition $\sigma\circ \nabla_\lambda:T_C\ra T_C$
is equal to $\lambda\cdot \id_{T_C}$ (where $\sigma: \widetilde
T_E\ra T_C$ is the natural projection). We denote by $\Loc_{G,\lambda}$ the stack of
$G$-bundles on $C$ with $\lambda$-connections. Then
\[\Loc_{G,1}=\Loc_G, \quad \Loc_{G,0}=\on{Higgs}_G.\]

Let $(E,\nabla_\lambda)\in\Loc_{G,\lambda}$. The $p$-curvature of
$\nabla_\lambda$ is defined as
$$\Psi(\nabla_\lambda): F^*T_{C'}\ra \ad(E), \quad v\ra\nabla_\lambda(v)^p-\lambda^{p-1}\nabla_\lambda(v^p).$$
The map
$\Loc_{G,\lambda}\to B_{\omega^p},\ (E,\nabla_\lambda)\mapsto
h_{\omega^p}(E,\Psi(\nabla_\lambda))$ factors through a unique map
$$h_{p,\lambda}:\Loc_{G,\lambda}\ra B^{(1)},$$
called the $p$-Hitchin map for $\lambda$-connections. It is clear
that $h_{p,1}=h_p$ and $h_{p,0}=F\circ h$, where $h:\on{Higgs}\to B$
is the usual Hitchin map and $F:B\to B^{(1)}$ is the relative Frobenius
of $B$. From this perspective, the $p$-Hitchin map can be regarded
as a deformation of the usual Hitchin map.


\subsection{Opers with marking}\label{oper}
In this subsection we recall the definition of opers with \emph{marking}
following \cite{B}.
There is a canonical decreasing Lie algebra filtration
$\{\fg^{k}\}$ of $\fg$
$$\cdot\cdot\cdot\supset\fg^{-1}\supset\fg^{0}\supset\fg^{1}\supset\cdot\cdot\cdot$$
such that $\fg^0=\frakb$, $\fg^1=\frakn$ and for any $i>0$ (resp.\
$<0$)  weights of the action of $\frakt=\on{gr}^0(\fg)$ on $\on{gr}^i(\fg)$
are sums of $i$ simple positive (resp.\ $-i$ simple negative) roots. In
particular, we have $\on{gr}^{-1}(\fg)=\oplus\fg_{\alpha}$, where
$\alpha$ is a simple negative root and $\fg_{\alpha}$ is the
corresponding root space.

Let $E$ be a $B_G$-torsor on $C$ and $E_G$ (resp. $E_T$) be the induced $G$-torsor
(resp. $T$-torsor)
on $C$. In this subsection, we denote by $\frakb_{E}$ and $\fg_{E_G}=\fg_E$  the associated
adjoint bundles. 
Let $\widetilde T_{E}$ and $\widetilde T_{E_G}$ be
the Lie algebroids of infinitesimal symmetries of  $E$ and $E_G$.
There is a natural embedding $\widetilde T_{E}\ra \widetilde
T_{E_G}$ and we have a canonical isomorphism
$$\widetilde T_{E_G}/\widetilde T_{E}\is(\fg/\fb)_{E}=: E\x^B(\fg/\fb).$$
For any connection $\nabla$ on $E_G$, we denote by $\bar\nabla$ the
composition
$$\bar\nabla:T_C\stackrel{\nabla}{\to}\widetilde
T_{E_G}\to  \widetilde T_{E_G}/\widetilde T_{E}\is(\fg/\fb)_{E}.$$

\begin{definition}\label{def: oper}
We fix a square root $\kappa=\omega^{1/2}$ of the canonical bundle $\omega$.
A $G$-oper on $C$ with marking is a triple $(E,\nabla,\phi)$ where $E$ is a
$B_G$-torsor on $C$,
$\nabla$ is a connection on $E_G$, and $\phi:E_T\is\omega^{1/2}\x^{\bG_m,2\rho}T$ is an isomorphism of $T$-torsor (we call $\phi$ the \emph{marking}), such that
\begin{enumerate}
 \item The image of $\bar\nabla$
lands in  $(\fg^{-1}/\frakb)_{E}\subset(\fg/\fb)_{E}$.
\item The composition $$T_C\stackrel{\bar\nabla}{\to}(\fg^{-1}/\frakb)
_{E}\stackrel{pr_{\alpha}}{\to}(\fg_{\alpha})_{E}$$ is an
isomorphism for every simple negative root $\alpha$. Here
$$pr_{\alpha}:(\fg^{-1}/\frakb)
_{E}=\oplus(\fg_{\beta})_{E}\ra(\fg_{\alpha})_{E}$$ is the natural
projection.
\item The condition (2) implies $\nabla$ induces an isomorphism
\begin{equation}\label{iso}
E_T\x^T(\fg^1/\fg^2)\is\oplus_{i=1}^l(\fg_{\alpha_i})_E
\stackrel{\tilde\phi}\ra\omega^{\oplus l}\is(\omega^{1/2}\x^{\bG_m,2\check\rho}T)\x^T(\fg^1/\fg^2).
\end{equation}
We require the marking $\phi$ to be compatible with $\tilde\phi$.
\end{enumerate}
\end{definition}

We denote by $\on{Op}_G$ the scheme of $G$-opers with marking on $C$.
Notice that if we drop condition (3) and the data of~$\phi$ in the above definition, then we obtain the definition of $G$-opers in \cite{BD}.
As shown in \emph{loc.\ cit.}, a $G$-oper has $Z(G)$ as its automorphism group, the additional
condition (3) eliminate these automorphisms (cf. \cite[Proposition 2.1]{B}).

\begin{remark}
When $G$ is of adjoint type, there exits a unique marking $\phi$
compatible with $\tilde\phi$. Thus, in this case, the condition (3)
is automatic. In general, the conditions (1) and (2) do not imply existence of $\phi$, hence
 we are limiting our collection of opers compared to \cite{BD}.
\end{remark}

\begin{example}
Consider the case $G=GL_n$. Then an oper with marking can be
described in terms of vector bundles as follows: it consists of the data $(E,\{E_i\}_{i=1,...,n},\nabla,\phi)$
where $E$ is a rank $n$ vector bundle on $C$,
$E_1\subset E_2\subset\cdot\cdot\subset E_n=E$ is a complete flag,
$\nabla$ is a connection on $E$,
and $\phi:E_1\is\omega^{(n-1)/2}$ is an isomorphism, such that
\begin{enumerate}
\item $\nabla(E_i)\subset E_{i+1}\otimes\omega$.
\item For each $i$, the induced morphism
$\gr_i(E)\stackrel{\gr_i(\nabla)}\ra \gr_{i+1}(E)\otimes\omega$ is an isomorphsim.
\end{enumerate}
\end{example}

\subsubsection{Affine space structure on $\on{Op}_G$}\label{a Op-aff}
We follow closely the presentation in \cite[3.1.7-3.1.9]{BD}.
We first consider the case $G=\SL_2$. 
Let $(E,\nabla,\phi)\in\on{Op}_{\SL_2}$.
The isomorphism~\eqref{iso} takes the form 
$(\frak n)_E\is E\times^B\frak n\is E_T\times^T(\frak g^1/\frak g^2)\is\omega$.
Translating $\nabla$ by a section of $\omega^{2}\is
(\frak n)_E\otimes\omega$, we get a new oper with marking 
and a direct computation shows that the resulting $\Gamma(C,\omega^{2} )$-action makes $\on{Op}_{\SL_2}$ a $\Gamma(C,\omega^2)$-torsor.

Let $G$ be general reductive group. 
We pick a principal embedding $\iota:\frak{sl}_2\imbed \fg$
and a homomorphism
 $\tilde \iota:\SL_2\to G$ with  $d\tilde \iota= \iota$. 
Let us write $B_0=T_0N_0$ for the standard Borel subgroup of $\SL_2$
and $\frak b_0=\frak t_0+\frak n_0$ the Borel subalgebra of $\on{sl}_2$.
Without loss of generality we can assume $\tilde \iota(B_0)\subset B$. 
Let $\{e_0,h_0,f_0\}$ be the standard basis of the principal 
$\frak{sl}_2\subset\fg$ and let $V=\fg^{e_0}$. 
The grading on $\fg=\oplus_i\fg_i$ induces a grading 
$V=\oplus_iV_i$ on $V$ and we consider 
the $\rho^+(\bG_m)$-action on $V$ given
by $t\cdot v_i=t^{i+1}v_i$ for $v\in V_i$.
We set $V_\omega=V\times^{\bG_m,\rho^+(\bG_m)}\omega$.
Then
the embedding 
$\frak n_0=\bC e_0\to V$ gives rise to an embedding
\begin{equation}\label{omega to V}
\omega^2\is (\frak n_0\times^{T_0}\omega^{1/2})\otimes\omega\to
(V\times^{T_0,\tilde \iota}\omega^{1/2})\otimes\omega\is V\times^{\bG_m,\rho^+(\bG_m)}\omega=V_\omega.
\end{equation}

Let $(E_0,\nabla_0,\phi_0)\in\on{Op}_{\SL_2}$.
The push-forward $\tilde \iota(E_0,\nabla_0,\phi_0)=(E,\nabla,\phi)$
along the principal homomorphism $\tilde \iota:\SL_2\to G$ is a $G$-oper with marking.
Moreover we have an embedding
\begin{equation}
V_\omega\is (V\times^{T_0,\tilde \iota}\omega^{1/2})\otimes\omega\is
(V\times^{B_0}E_0)\otimes\omega\to (\fb\times^BE)\otimes\omega=(\frak b)_E\otimes\omega.
\end{equation}
Thus translating $\nabla$ by a section of $v\in
\Gamma(C,V_\omega)\subset \Gamma(C,(\frak b)_E\otimes\omega)$  defines a new $G$-oper with marking denoted by 
$\nabla+v$.
Consider the following $\Gamma(C,V_\omega)$-torsor
\[\on{Op}_{\SL_2}\times^{\Gamma(C,\omega^2)}\Gamma(C,V_\omega),\]
where $\Gamma(C,\omega^2)$ acts on $\Gamma(C,V_\omega)$
via the map~\eqref{omega to V}.

\begin{lem}\label{affine}
The principal $\tilde \iota:\SL_2\to G$  gives rise to an isomorphism
\[\on{Op}_{\SL_2}\times^{\Gamma(C,\omega^2)}\Gamma(C,V_\omega)\is\on{Op}_G\ \ \ 
(\nabla_0,v)\to \nabla+v.\]
\end{lem}
\begin{proof}
It suffices to check the assertion locally.
So it is enough to show that 
for any $(E,\nabla,\phi)\in\on{Op}_G$, there exists a unique trivialization 
of $E$ such that the connection 
has the form
$\nabla=d+dx\otimes (f_0+v)$
where $v:C\to V$.
This follows from  Kostant's theorem, see
\cite[Section 3.4]{BDop} or \cite[Lemma 3.5]{B}.\footnote{In \emph{loc. cit.}
the authors assume $G$ is adjoint, but since Kostant's theorem holds for general reductive groups, 
 the same proof applies to the general case.}
\end{proof}

\subsubsection{Filtration on $\mO(\on{Op}_G)$}
One defines a $(G,\lambda)$-oper with marking as before by replacing connection
$\nabla$ by $\lambda$-connection $\nabla_\lambda$. We denote by
$\on{Op}_{G,\lambda}$ the scheme of $(G,\lambda)$-opers with marking.
Clearly we have  $\on{Op}_{G,1}=\on{Op}_G$.

All $(G,\lambda)$-opers with marking form a scheme equipped with a morphism
$q: \widetilde{\on{Op}_G}\ra\mathbb A^1$, such that the fiber of
$\widetilde{\on{Op}_G}$ over $\lambda\in\bbA^1(k)$ is
$\on{Op}_{G,\lambda}$.
Moreover, there a $\bG_m$-action on
$\widetilde{\on{Op}_G}$, given by $(E,\nabla)\mapsto (E,t\nabla)$ and the morphism $q$ is
$\bG_m$-equivariant.
In addition, if we let $\mathrm{Bun}_N^{\kappa}$ denote the moduli of $B_G$-bundles $E$ on $C$ (say $C$ projective) equipped with an isomorphism $\phi: E_T\is\omega^{1/2}\x^{\bG_m,2\rho}T$ as in Definition \ref{def: oper}, then there is a natural morphism $f:\widetilde{\on{Op}_G}\to \mathrm{Bun}_N^{\kappa}$. When $G=\SL_2$, there are canonical isomorphisms $ \mathrm{Bun}_N^{\kappa}\cong H^1(C, \omega) \cong \mathbb A^1$ under which $f$ is identified with $q$.

By the same discussions as in Lemma \ref{affine} and the discussions before that, there is an action of $\Gamma(C, \omega^2)$ on $\widetilde{\on{Op}_{\SL_2}}$, and the map $f$ (and $q$) is $\Gamma(C, \omega^2)$-equivariant. In addition, there is a canonical isomorphism
\[\widetilde{\on{Op}_{\SL_2}}\times^{\Gamma(C,\omega^2)}\Gamma(C,V_\omega)\is\widetilde{\on{Op}_G},\ \ \ 
(\nabla_\lambda,v)\to \iota_*(\nabla_\lambda)+v.\]


\begin{lem}
The map $q:\widetilde{\on{Op}_G}\ra\mathbb A^1$ is flat.
\end{lem}
\begin{proof}
We will actually prove that the morphism is smooth.
By the above isomorphism, it is enough to prove the statement for $G=\SL_2$.

Recall that a morphism $f:X\to Y$ of finite type smooth schemes over an algebraically closed field $k$ is smooth if and only if for every closed point $x$ of $X$, the induced map of tangent spaces $df: T_xX\to T_{f(x)}Y$ is surjective. 
Now we compute the tangent map in our situation.  Since the map is $\Gm$-equivariant and has smooth fibers
over $\lambda\ne 0$ by Lemma \ref{affine}, it is smooth over $\lambda\ne 0$. Thus
 it is enough to compute it  at  points of $\widetilde{\on{Op}_{\SL_2}}$ lying over $\lambda=0$. Let $(E_0,\nabla_0)$ be such a point (so $\nabla_0$ is a Higgs field). Then a standard deformation theory argument shows that the tangent space of $\widetilde{\on{Op}_{\SL_2}}$ at this point is $H^1(C, \omega)\oplus H^0(C,\omega^2)$ and the differential $dq$ is identified with the projection to the first factor $H^1(C,\omega)=k$. The lemma is proved.
 \end{proof}

We have the forgetful map $\on{Op}_{G,\lambda}\ra\on{LocSys}_{G,\lambda},\ (E,\nabla_\lambda,\phi)\ra (E_G,\nabla_\lambda)$ and at $\lambda=0$
\begin{equation}\label{Kostant map}
B\is\on{Op}_{G,0}\ra\Loc_{G,0}=\on{Hggs}_G
\end{equation}
is the Kostant section $\epsilon_\kappa$ induced by $\kappa=\omega^{1/2}$.

The $p$-Hitchin map for $\lambda$-connections gives
$$\tilde\pi_p:\widetilde{\on{Op}_G}\ra B^{(1)}\x\mathbb A^1,\ \ \
\ (E,\nabla_\lambda,\phi)\ra
(h_{p,\lambda}(E_G,\nabla_\lambda),\lambda).$$
The map $\tilde\pi_p$ is compatible with the natural morphisms to $\mathbb A^1$ from both sides (namely the morphism $q$ on the left and the second projection on the right), and is
$\bG_m$-equivariant where $\bG_m$ acts diagonally on $B^{(1)}\x\mathbb A^1$.
We denote by $\pi_{p,\lambda}:\Op_{G,\lambda}\ra B^{(1)}$ the base change of
$\tilde\pi_p$ to $\lambda\in\mathbb A^1(k)$. When
$\lambda=1$, we get a map
\begin{equation}\label{pi_p}
\pi_{p}:=\pi_{p,1}:\Op_G\ra\Loc_G\xrightarrow{h_p} B^{(1)},
\end{equation}
and  (\ref{Kostant map}) implies $\pi_{p,0}:B=\on{Op}_{G,0}\ra B^{(1)}$ is the relative Frobenius morphism $F:B\ra B^{(1)}$.

%

We recall the well-known equivalence of categories between $k$-algebras equipped with an exhaustive filtration and graded flat $k[\lambda]$-algebras (with $\deg \lambda=1$). Namely, the equivalence sends a  $k$-algebra $A$ equipped with an increasing filtration $F_\bullet A$ such that $k\subset F_0A$, $\cup_i F_iA=A$, to the graded flat $k[\lambda]$-algebra $R= \oplus_i F_i A \lambda^i$. The quasi-inverse functor is given by sending $R=\oplus R_i$ to $A=R/(\lambda-1)R$ with $F_iA=\mathrm{Im}(R_i\mapsto A)$. Note that under this equivalence, $R/\lambda R$ can be identified with the associated graded of $A$.
Now the discussion above imply the following lemma.
\begin{lem}\label{gr}
Let $\pi_p^*:\mO({B^{(1)}})\to \mO({\Op_G})$ be the map of ring of functions
corresponding to $\pi_{p}:\on{Op}_G\ra B^{(1)}$. Then there are
filtrations on $\mO({\Op_G})$ and $\mO({B^{(1)}})$ such that
\begin{enumerate}
\item The associated graded $\gr(\mO({\Op_G}))\is\mO(B)$ and $\gr(\mO({B^{(1)}}))\is \mO({B^{(1)}})$.
\item $\pi_p^*$ is compatible with the filtrations.
\item The induced morphism
$$\gr(\pi_p^*): \mO({B^{(1)}})\ra\mO({B})$$
is the relative Frobenius map.
\end{enumerate}
\end{lem}

We need one more property of the above filtration on $\on{Op}_G$.

\begin{lem}
The exhaustive filtration on $\on{Op}_G$ is separated. In fact, $F_{-1} \mO(\on{Op}_G)=0$.
\end{lem}
\begin{proof}
We claim that the degree $-1$ part of $\mO(\widetilde{\Op_G})$ is zero. First, we know that the $\mathbb G_m$-action on $\Op_{G,0}$ is just by dilation of the vector space. So $\gr \mO(\on{Op}_G)$ is concentrated in non-negative degrees. It follows that 
\[
\cdots \stackrel{\cdot \lambda}{\cong}  \mO(\widetilde{\Op_G})^{-2}\stackrel{\cdot \lambda}{\cong}  \mO(\widetilde{\Op_G})^{-1} \stackrel{\cdot \lambda}{\hookrightarrow}  \mO(\widetilde{\Op_G})^{0}  \stackrel{\cdot \lambda}{\hookrightarrow}  \mO(\widetilde{\Op_G})^{1}  \stackrel{\cdot \lambda}{\hookrightarrow}  \cdots.
\]
We let $\mO(\widetilde{\Op_G})^{\mathrm{deg}}\subset \mO(\widetilde{\Op_G})$  be the graded subspace with $\mO(\widetilde{\Op_G})^{\mathrm{deg},m}=\lambda^{m+n} \mO(\widetilde{\Op_G})^{-n}$ for $n\gg 0$. Then $\mO(\widetilde{\Op_G})^{\mathrm{deg}}$ is a graded ideal, defining a flat $\mathbb G_m$-invariant closed subscheme $\widetilde{\Op_G}^{\mathrm{nondeg}}\subset \widetilde{\Op_G}$ satisfying $\widetilde{\Op_G}^{\mathrm{nondeg}}|_0=\widetilde{\Op_G}|_0$. As $\widetilde{\Op_G}|_1=\Op_G$ is smooth irreducible and $\dim(\widetilde{\Op_G}|_1)=\dim(\widetilde{\Op_G}|_0)$, we must have $\widetilde{\Op_G}^{\mathrm{nondeg}}=\widetilde{\Op_G}$. It follows that $\mO(\widetilde{\Op_G})^{-1}=0$.
\end{proof}

\subsection{Proof of Theorem \ref{a main}}
Let $\pi_p:\Op_G\ra B^{(1)}$ be the map in (\ref{pi_p}).
We first show that $\pi_p$
is finite and surjective. Thus we need to show that
$\pi_p^*:\mO({B^{(1)}})\ra \mO({\Op_G})$ is injective and $\mO({\Op_G})$ is finitely
generated as an $\mO({B^{(1)}})$-module. Since both rings $\mO({\Op_G})$ and
$\mO({B^{(1)}})$ are filtered and $\pi_p^*$ is compatible with the
filtrations, it is enough to show that the associated graded map
$\gr(\pi_p^*):\gr(\mO({B^{(1)}}))\ra \gr(\mO({\Op_G}))$ is injective and $\gr(\mO({\Op_G}))$
is a finitely generated $\gr(\mO({B^{(1)}}))$-module. But this is clear since
by the lemma above $\gr(\pi_p^*)$ is the Frobenius map. Now $\pi_p$ is a
finite map between $\on{Op}_{G}$ and $B^{(1)}$, which are smooth of the
same dimension, and therefore it is flat. In addition, as the relative Frobenius map $B\to B^{(1)}$ is of degree $p^{\dim B}$, so is $\pi_p$.

\begin{remark}\label{rem_HW}
Lemma \ref{gr} shows that the map $\pi_p:\Op_G\to B^{(1)}$ is a deformation of the Frobenius
morphism $Fr:B\to B^{(1)}$. In the special case when $G=GL(1)$ it is not hard to see that
one can identify $\Op_{GL(1)}$ with $B=\Gamma(C,\omega)$ so that the morphism $\pi_p$ is identified
with $Fr-\mathfrak C$ where $\mathfrak C: \Gamma(C,\omega)\to \Gamma(C,\omega)^{(1)}$ is the map induced by Cartier
isomorphism; it is closely related to the Hasse-Witt matrix of $C$. In particular, $\pi_p$ is purely inseparable if and only if $C$ is supersingular. It would be interesting to obtain a similar explicit description of the map $\pi_p$ for nonabelian $G$.
\end{remark}

\begin{remark}\label{rem_EFK}
There is a  parallel between analytic
constructions involving a complex algebraic variety $X_{\mathbb C}$ and algebraic constructions
involving  an algebraic variety $X_k$
over a field $k$ of characteristic $p>0$. Thus the analytic exponential function is analogous 
to the Artin-Schreier polynomial $x^p-x$ and the Stone -- von Neumann Theorem is 
parallel to the Azumaya property of crystalline differential operators:
the former asserts that unique
 representation of a Heisenberg group on a Hilbert space is realized as $L^2({\mathbb R}^n)$,
 while the latter implies that unique irreducible representation of the ring of differential operators
on ${\mathbb A}^n$ with a fixed central character is realized as the space of functions on the
Frobenius neighborhood of a point $\O(FrN(x))$, $x\in {\mathbb A}^n$, see \cite[Remark 2.2.4(3)]{BMR}.

Another parallel appearing in the literature is between the nonabelian Hodge theory
for local systems on a complex curve and the (partially defined) Cartier transform \cite{OV}: both construction
produce a Higgs field starting from a local system on the curve.  Notice that
 the support of the Cartier transform (when defined) of a $D$-module coincides with the support of the $D$-module as a module over the $p$-center.

In \cite{EFK} one finds a conjectural description of the spectrum  $S_{\mathbb C}$ of the ring of global
twisted differential operators on $\Bun_G$ for a complex curve acting on the space
of $L^2$ sections  of $\Omega^{1/2}_{\Bun}$. The above observation connecting
$L^2({\mathbb R}^n)$ to $\O(FrN(x))$ suggests that $S_{\mathbb C}$ is analogous
to $S_k$, the spectrum of the action of the ring $\Gamma(\DBu)$ acting on the space
of sections of $\Gamma(\Omega_{\Bun}^{1/2},FrN(x))$, $x\in \Bun$.
Notice that this module has zero $p$-curvature, so Theorem \ref{a main}
implies that $S_k$ is a subset in $\pi_p^{-1}(0)$, the set of opers with zero $p$-curvature
(dormant opers in the terminology of \cite{Mo}).
The analogy described in the previous paragraph suggests that the corresponding characteristic
zero object should be the set of opers which under the nonabelian Hodge theory corresponds to a Higgs field $(E,\phi)$
with $\phi=0$. A local system corresponding to such a Higgs field under the nonabelian Hodge theory 
has {\em unitary} monodromy. On the other hand, \cite[Conjecture 1.11]{EFK} asserts that $S_{\mathbb C}$
is a subset in the set of opers with {\em real} monodromy. Thus the above heuristics agrees with the 
Conjecture of \cite{EFK} up to the change of the real form of the complex group $G_{\mathbb C}$.

\end{remark}
\end{appendix}

\ncmd\eml{{\em Email:} \tt}

\bigskip \parskip=\medskipamount

\footnotesize\parindent=0pt

 {\bf R.B.}: Department of Mathematics, Massachusetts Institute
of Technology, Cambridge MA 02139, USA\\
{\eml bezrukav@math.mit.edu}

 {\bf R.T.}: 
   Center for Advanced Studies,
   Skolkovo Institute of Science and Technology, Moscow, Russia\\
{\eml travkin@alum.mit.edu}

{\bf T.-H.C.}: School of Mathematics, University of Minnesota, Minneapolis, Vincent Hall, MN, 55455\\
{\eml chenth@umn.edu}

{\bf X.Z.}: Department of Mathematics, California Institute of Technology,
        1200 E. California Blvd., Pasadena, CA 91125, USA\\
        {\eml xzhu@caltech.edu}

\end{document}